\documentclass{article}
\usepackage[left=3cm,right=3cm,top=2.5cm,bottom=2.5cm,includeheadfoot]{geometry}
\usepackage{etoolbox}
\usepackage{graphicx}
\usepackage{amsmath,amssymb, amsthm}
\usepackage{dsfont}
\usepackage{enumitem}
\usepackage{tikz}
\usetikzlibrary{external,calc,positioning,arrows.meta,quotes}

\usepackage{mathtools}
\usepackage{multirow}
\usepackage{algorithm2e}
\usepackage{bm}
\usepackage{hyperref}

\newtheorem{definition}{Definition}
\newtheorem{remark}{Remark}
\newtheorem{lemma}{Lemma}
\newtheorem{proposition}{Proposition}
\newtheorem{theorem}{Theorem}

\newcommand{\R}{\mathbb{R}}
\newcommand{\N}{\mathbb{N}}
\newcommand{\Id}{{\mathds{1}}}
\newcommand{\tr}{\mathrm{tr}}
\newcommand{\domain}{\Omega}
\newcommand{\MN}{{\mbox{\tiny $M\!N$}}}
\newcommand{\discreteDomain}{\Omega_{\MN}}
\renewcommand{\div}{{\mathrm{div}}}
\newcommand{\GL}{{\mathrm{GL}}}
\newcommand{\sym}{\mathrm{sym}}
\DeclareMathOperator*{\argmin}{argmin}

\def\Xint#1{\mathchoice
{\XXint\displaystyle\textstyle{#1}}
{\XXint\textstyle\scriptstyle{#1}}
{\XXint\scriptstyle\scriptscriptstyle{#1}}
{\XXint\scriptscriptstyle\scriptscriptstyle{#1}}
\!\int}
\def\XXint#1#2#3{{\setbox0=\hbox{$#1{#2#3}{\int}$}
\vcenter{\hbox{$#2#3$}}\kern-.5\wd0}}
\def\dashint{\Xint-}
\newcommand\restr[2]{{\left.\kern-\nulldelimiterspace #1 \vphantom{\big|}\right|_{#2}}}

\newcommand{\energyDensity}{\mathrm{W}}
\newcommand{\energy}{\mathcal{W}} 

\newcommand{\pathenergy}{\mathcal{E}}
\newcommand{\Pathenergy}{\mathbf{E}}
\newcommand{\FeatureOperator}{\mathbf{F}}
\newcommand{\image}{u}
\newcommand{\Image}{\mathbf{u}}
\newcommand{\feature}{f}
\newcommand{\discreteFeature}{\mathbf{f}}
\newcommand{\discreteAnisotropy}{\mathbf{a}}

\newcommand{\imageSpace}{\mathcal{I}}
\newcommand{\featureSpace}{\mathcal{F}}

\newcommand{\motionSpace}{\mathcal{V}}
\newcommand{\admset}{\mathcal{A}} 
\newcommand{\discreteAdmset}{\mathcal{A}_{\MN}}

\newcommand{\interpolfeature}{\mathcal{F}}

\newcommand{\featurevec}{\mathbf{f}}

\newcommand{\featureinnervec}{{\mathbf{\hat f}}}

\newcommand{\featureinnervectest}{{\mathbf{\hat g}}}
\newcommand{\defvec}{\mathbf{\Phi}}
\newcommand{\deformation}{\phi}
\newcommand{\discreteDeformation}{\bm{\phi}}
\newcommand{\warp}{\mathbf{T}}
\newcommand{\DataEnergy}{\mathbf{D}}
\newcommand{\RegEnergy}{\mathbf{R}}
\newcommand{\conv}{\mathrm{conv}}
\newcommand{\dx}{\,\mathrm{d}}
\newcommand{\anisotropy}{a}
\newcommand{\proj}{\mathcal{P}}

\def\Xint#1{\mathchoice
{\XXint\displaystyle\textstyle{#1}}
{\XXint\textstyle\scriptstyle{#1}}
{\XXint\scriptstyle\scriptscriptstyle{#1}}
{\XXint\scriptscriptstyle\scriptscriptstyle{#1}}
\hspace{-1.5pt}\int}
\def\XXint#1#2#3{{\setbox0=\hbox{$#1{#2#3}{\int}$}
\vcenter{\hbox{$#2#3$}}\kern-.5\wd0}}
\def\dashint{\Xint-}

\begin{document}

\title{Image Morphing in Deep Feature Spaces: Theory and Applications}

\author{
Alexander Effland\thanks{Institute of Computer Graphics and Vision, Graz University of Technology (\url{alexander.effland@icg.tugraz.at}, \url{erich.kobler@icg.tugraz.at}, \url{pock@icg.tugraz.at})} \and Erich Kobler\footnotemark[1] \and Thomas Pock\footnotemark[1]
\and Marko Rajkovi\'c\thanks{Institute for Numerical Simulation, University of Bonn (\url{marko.rajkovic@ins.uni-bonn.de}, \url{martin.rumpf@ins.uni-bonn.de})} \and Martin Rumpf\footnotemark[2]}

\maketitle
\begin{abstract}
This paper combines image metamorphosis with deep features. 
To this end, images are considered as maps into a high-dimensional feature space and
a structure-sensitive, anisotropic flow regularization is incorporated in the metamorphosis model proposed by Miller, Trouv\'e, Younes and coworkers \cite{MiYo01,TrYo05}.
For this model a variational time discretization of the Riemannian path energy is presented and the existence of discrete geodesic paths minimizing this energy is demonstrated.
Furthermore, convergence of discrete geodesic paths to geodesic paths in the time continuous model is investigated.
The spatial discretization is based on a finite difference approximation in image space and a stable spline approximation in deformation space, the fully discrete model is optimized using the iPALM algorithm.
Numerical experiments indicate that the incorporation of semantic deep features is superior to intensity-based approaches\footnote{
This publication is an extended version of the previous conference proceeding~\cite{EfKo19} presented at SSVM 2019.}.

\end{abstract}

\section{Introduction}
\label{Introduction}
In mathematical imaging, image morphing is the problem of computing a visually appealing transition of two images such that
semantically corresponding regions are mapped onto each other.
A well-known approach for image morphing is the metamorphosis model originally introduced 
by Miller, Trouv\'e, and Younes~\cite{MiYo01,TrYo05,TrYo05a}, which generalizes
the flow of diffeomorphism model and the large deformation diffeomorphic metric mapping (LDDMM)
which dates back to the pioneering work of Arnold \cite{Ar66a} with its exploration and extension in imaging by 
Dupuis, Grenander and others \cite{DuGrMi98,BeMiTr05,JoMi00,MiTrYo02,VS09,VRRC12}.
From the perspective of the flow of diffeomorphism model,
each point of the reference image is transported to the target image in
an energetically optimal way such that the image intensity is preserved along the trajectories of the pixels.
Here, the energy measures the total dissipation of the underlying flow.
The metamorphosis model additionally allows for image intensity modulations along the trajectories
by incorporating the magnitude of these modulations, which is reflected by the integrated squared material derivative of image trajectories
as a penalization term in the energy functional.
Recently, the metamorphosis model has been extended to images on Hadamard manifolds~\cite{NePe18,EfNe19}, to reproducing kernel Hilbert spaces~\cite{RY16},
to functional shapes~\cite{CCT16} and to discrete measures~\cite{RY13}.
For a more detailed exposition of these models we refer the reader to \cite{Younes2010,MTY15} and the references therein.

Starting from the general framework for variational time discretization in geodesic calculus~\cite{RuWi12b},
a variational time discretization of the metamorphosis model for square-integrable images $L^2(\domain,\R^n)$ was proposed in~\cite{BeEf14}.
Moreover, the existence of discrete geodesic paths as well as the Mosco--convergence of the time discrete to the time continuous metamorphosis model was proven.
However, the classical metamorphosis model, its time discrete counterpart and the spatial discretization based on finite elements in~\cite{BeEf14} exhibit several drawbacks: 
\begin{itemize}
\item[-]
The comparison of images in their original gray- or color space is not invariant to natural radiometric transformations caused by
lighting or material changes, shadows etc.~and hence might lead to
a blending along the discrete geodesic path instead of flow-induced geometric transformations.
\item[-]
Texture patterns, which are important for a natural appearance of images, are often destroyed along the geodesic path due to the color-based matching.
\item[-]
Sharp interfaces such as object boundaries, which frequently coincide with depth discontinuities of a scene, are in general not preserved
along a geodesic path because of the strong smoothness implied by the homogeneous and isotropic variational prior for the 
deformation fields.
\end{itemize}
To overcome these problems originating from the inten\-sity-based matching, we propose a multiscale feature space approach incorporating the deep convolution neural network introduced in~\cite{SiZi15}. 
In detail, this convolutional neural network, which was trained to classify the ImageNet dataset~\cite{KrSu12}, extracts semantic features using $19$~weight layers, each composed of small $3\times 3$-convolu\-tion filters with subsequent nonlinear ReLU activation functions.
This network defines a feature extraction operator, where each feature map is considered as a continuous map into some higher-dimensional feature space 
consisting of vectors in $\R^C$, where $C$ ranges from $64$ to $512$ depending on the considered scale associated with a certain network layer. Throughout the paper we refer to this network as VGG network ("Visual Geometry Group in Oxford").
Compared to the original time discrete metamorphosis model~\cite{BeEf14} we advocate a metamorphosis model in a 
deep feature space, which amounts to replacing the input images by feature vectors combining image intensities and semantic information 
generated by the feature extraction operator.
To explicitly allow for discontinuities in the deformation fields, we introduce an anisotropic regularization of the time discrete deformation sequence.
Since motion discontinuities and object interfaces in images commonly coincide, the considered anisotropy solely depends on the magnitude of image gradients.

We prove the existence of discrete geodesic paths for the deep feature metamorphosis model and discuss its Mosco--convergence to  the appropriate time continuous metamorphosis model in deep feature space. This in particular implies the convergence of time discrete to time continuous geodesic paths and establishes the existence of time continuous geodesics as minimizers of the time continuous metamorphosis model.

We propose a finite difference/third order B-spline discretization for the fully discrete feature space metamorphosis model and use the iPALM algorithm~\cite{PoSa16} for the optimization,
which leads to an efficient and robust computation of morphing sequences that visually outperform the prior intensity-based finite element discretization discussed in~\cite{BeEf14}.
This scheme is significantly less sensitive to intensity modulations due to the exploitation of semantic information.

Note that this publication is an extended version of the conference proceeding~\cite{EfKo19}, in which
the model is adapted and in addition a rigorous mathematical analysis of this novel model is presented.
In fact, the morphing sequence is no longer retrieved in a post-processing step.
Instead, the color values are part of the feature vector. 
Different from the prior proceedings article, we prove the existence of time discrete geodesics in feature space, 
present a time continuous model and discuss the issue of convergence of the discrete functionals. 

\paragraph{Notation.}
Throughout this paper, we assume that the image domain~$\domain\subset\R^n$ for $n\in\{2,3\}$ is bounded and strongly Lipschitz.
We use standard notation for Lebes\-gue and Sobolev spaces from the image domain~$\domain$ to a Banach space~$X$,
i.e.~$L^p(\domain,X)$ and $H^m(\domain,X)$ and omit $X$ if the space is clear from the context.
The associated norms are denoted by $\Vert\cdot\Vert_{L^p(\domain)}$ and $\Vert\cdot\Vert_{H^m(\domain)}$, respectively, and the seminorm in $H^m(\domain)$ is given by $|\cdot|_{H^m(\domain)}$, i.e.
\[
|f|_{H^m(\domain)}=\Vert D^m f\Vert_{L^2(\domain)}\,,\quad
\Vert f\Vert_{H^m(\domain)}^2=\sum_{j=0}^m|f|_{H^j(\domain)}^2
\]
for $f\in H^m(\domain)$.
We use the notation $C^{k,\alpha}(\overline \domain,X)$ for H\"older spaces of order $k\geq0$ with regularity $\alpha\in(0,1]$, the corresponding (semi)norm is
\begin{equation*}
|f|_{C^{0,\alpha}(\overline\domain)}=\sup_{x\neq y\in\domain}\frac{|f(x)-f(y)|}{|x-y|^\alpha}\,, \quad
\Vert f\Vert_{C^{k,\alpha}(\overline\domain)}=\Vert f\Vert_{C^k(\overline\domain)}+\sum_{|\beta|=k}|D^{\beta}f|_{C^{0,\alpha}(\overline \domain)}\,.
\end{equation*}
The symmetric part of a matrix~$A\in\R^{n,n}$ is denoted by $A^\sym$, i.e.~$A^\sym=\frac{1}{2}(A+A^\top)$ and the symmetrized Jacobian of a differentiable function $\phi$ by $\varepsilon[\phi]=(D\phi)^{\sym}$.
We denote by $\GL^+(n)$ the elements of $\GL(n)$ with positive determinant, and by $\Id$ both the identity map and the identity matrix.
Finally, $\dot{f}$ refers to the temporal derivative of a differentiable function~$f$.

\paragraph{Organization.}
This paper is structured as follows: 
in Section~\ref{sec:metamorphosis}, we review the classical metamorphosis model and present its extension to deep feature spaces.
Then, in Section~\ref{sec:timeDiscrete} we introduce the time discrete deep feature metamorphosis model and prove the existence of geodesic paths.
In Section~\ref{sec:Mosco}, we present a time continuous metamorphosis model and comment on the Mosco--convergence in deep feature space.
The fully discrete model and the optimization scheme using the iPALM algorithm are presented in Section~\ref{sec:fullyDiscrete}.
Finally, in Section~\ref{sec:results} several examples demonstrate the applicability of the proposed methods to real image data.

\section{Metamorphosis model}\label{sec:metamorphosis}
In this section, we briefly review the classical flow of diffeomorphism model and the metamorphosis model as its generalization.
Then, we extend the metamorphosis model to the space of deep features, where we additionally incorporate an anisotropic regularization.

\subsection{Flow of diffeomorphism}
In what follows, we present a very short exposition of the flow of diffeomorphism model 
and refer the reader to~\cite{DuGrMi98,BeMiTr05,JoMi00,MiTrYo02} for further details.
In the flow of diffeomorphism model, the temporal change of image intensities is determined by a \emph{family of diffeomorphisms} $(\psi(t))_{t\in [0,1]}:\overline\domain\to\R^n$ 
describing a flow transporting image intensities along particle paths.
The main assumption of this model is the \emph{brightness constancy assumption}, which
is equivalent to a vanishing material derivative~$\frac{D}{\partial t}\image=\dot\image+v\cdot D\image$
along a path $(\image(t))_{t\in [0,1]}$ in the space of images, where $v(t)=\dot\psi(t)\circ\psi^{-1}(t)$ denotes the time-dependent \emph{Eulerian velocity}.
The Riemannian space of images is endowed with the following metric and path energy
\begin{equation*}
g_{\psi_t}(\dot{\psi}_t,\dot{\psi}_t)=\int_\domain L[v,v]\dx x\,, \quad
\pathenergy_{\psi_t}[(\psi_t)_{t\in [0,1]}]=\int^1_0 g_{\psi_t}(\dot{\psi}_t,\dot{\psi}_t)\dx t\,.
\end{equation*}
Note that we use $\psi_t$ as a shortcut for the function $x\mapsto\psi(t,x)$.
Here, the quadratic form $L$ is the higher order elliptic operator
\begin{equation*}
L[v,v]= \frac{\lambda}{2} (\tr \varepsilon[v])^2 + \mu \tr (\varepsilon[v]^2) + \gamma |D^m v|^2,
\end{equation*}
where $m>1+\frac{n}{2}$ and $\lambda,\mu,\gamma >0$.
Physically, the metric $g_{\psi_t}(\dot \psi_t,\dot \psi_t)$ describes the viscous dissipation in a multipolar fluid model as investigated by Ne\v{c}as and \v{S}ilhav\'y~\cite{NeSi91}.
The first two terms of the integrand represent the dissipation density in a Newtonian fluid and the third term can be regarded as a higher order measure for friction.
Following~\cite[Theorem 2.5]{DuGrMi98}, paths with a finite energy, which connect two diffeomorphisms $\psi_0=\psi_A$ and $\psi_1=\psi_B$, 
are actually one-parameter families of diffeomorphisms.
Given two image intensity functions $\image_A,\image_B\in L^2(\domain)$, an associated geodesic path is a family of images $\image=(\image(t):\domain\to \R)_{t\in [0,1]}$ with 
$\image(0,\cdot)=\image_A(\cdot)$ and $\image(1,\cdot)=\image_B(\cdot)$, which minimizes the path energy.
The resulting flow of images is given by $\image(t,\cdot)=\image_A\circ\psi^{-1}_t(\cdot)$.

\subsection{Metamorphosis model in image space}\label{sub:metamorphosis}
The metamorphosis approach originally proposed by Miller, Trouv\'e, Younes and coworkers in~\cite{MiYo01,TrYo05,TrYo05a}
generalizes the flow of diffeomorphism model by allowing for image intensity variations along motion paths 
and penalizing the squared material derivative in the metric. 
Under the assumption that the image path~$\image$ is sufficiently smooth, the metric and the path energy read as
\begin{equation*}
g(\dot\image,\dot\image)=\min_{v:\overline\domain\to\R^n}\int_\domain L[v,v]+\frac1\delta z^2\dx x\,, \quad
\pathenergy[\image]=\int_0^1 g(\dot\image(t),\dot \image(t))\dx t
\end{equation*}
for a penalization parameter $\delta>0$, where $z=\frac{D}{\partial t}\image=\dot\image+v\cdot D\image$ denotes the material derivative of $\image$.
The Lagrangian formulation of this variation of the image intensity along motion trajectories can be phrased as follows:
for all $s,t\in[0,1]$ we have 
\begin{align}\label{eq:Lagrange}
\image(t,\psi_t)-\image(s,\psi_s)=\int_s^t z(r,\psi_r)\dx r\,.
\end{align}
Hence, the flow of diffeomorphism model is the limit case of the metamorphosis model for~$\delta\to 0$.
This definition of the metric has two major drawbacks:
In general, paths in the space of images do not exhibit any smoothness properties (neither in space nor time),
and therefore the evaluation of the material derivative is not well-defined.
Moreover, since different pairs~$(v,\frac{D}{\partial t}\image)$ of velocity fields and material derivatives
can imply the same time derivative of the image path~$\dot\image$, the restriction to equivalence classes of pairs is required,
where two pairs are equivalent if and only if they induce the same temporal change of the image path~$\dot\image$.

To tackle both problems, Trouv\'e and Younes~\cite{TrYo05a} proposed
a nonlinear geometric structure in the space of RGB images~$\imageSpace\coloneqq L^2(\domain,\R^3)$.
In detail, for a given image path~$\image\in L^2([0,1],\imageSpace)$ and an associated
velocity field~$v\in L^2((0,1),\motionSpace)$, where $\motionSpace \coloneqq H^{m}(\domain,\R^n)\cap H^{1}_0(\domain,\R^n)$ denotes the velocity space,
the \emph{weak material derivative} $z\in L^2((0,1),L^2(\domain,\R^3))$ is incorporated in the model, which is implicitly given by
\begin{align}\label{eq:weakz}
\int_0^1\int_\domain\eta z\dx x\dx t=-\int_0^1\int_\domain(\partial_t\eta+\div(v\eta))\image\dx x\dx t
\end{align}
for a smooth test function $\eta\in C^{\infty}_c((0,1)\times\domain)$. 
We consider $(v,z)$ as a tangent vector in the tangent space of $\imageSpace$ at the image $\image$ and write $(v,z) \in T_\image\imageSpace$ defined by~\eqref{eq:weakz}.
Indeed, $(v,z)$ represents a variation of the image $\image$ via transport and 
change of intensity.
This (weak) formulation and the consideration of equivalence classes of motion fields and material derivatives inducing the same temporal change of the image intensity 
gives rise to the notion $H^1([0,1],\imageSpace)$ for regular paths in the space of images.
For details we refer the reader to~\cite{TrYo05a}.
The \emph{path energy in the metamorphosis model}
for a regular path $\image\in H^1([0,1],\imageSpace)$ is then defined as
\begin{equation}
\pathenergy[\image]=\int_0^1\inf_{(v,z)  \in T_u\imageSpace} \int_\domain L[v,v]+\frac{1}{\delta}z^2\dx x\dx t\,.
\label{eq:DefinitionPathenergyImage}
\end{equation}
Then, image morphing of two input images~$\image_A,\image_B\in\imageSpace$ amounts to computing a shortest geodesic path~$\image\in H^1([0,1],\imageSpace)$ in the metamorphosis model,
which is defined as a minimizer of the path energy in the class of regular curves such that $\image(0)=\image_A$ and $\image(1)=\image_B$.
The existence of a shortest geodesic is proven in~\cite[Theorem 6]{TrYo05a}.
Note that the infimum in~\eqref{eq:DefinitionPathenergyImage} is attained, which is shown in~\cite[Proposition~1 \& Theorem~2]{TrYo05a}.

\subsection{Metamorphosis model in deep feature space}
In this subsection, we extend the metamorphosis model to images as maps into a deep feature space
with the aim to increase the reliability and robustness of the resulting morphing.
To further improve the quality of the deformations, we incorporate an aniso\-tropic regularization of the deformation field.
We will 
compute geodesic paths in the \emph{feature space~$\featureSpace\coloneqq L^2(\domain,\R^{3+C}$}) for $C\geq 0$.
Here, the first part~$\image\in\imageSpace$ of a feature vector~$\feature=(\image,\tilde\feature)\in\featureSpace$ encodes the RGB image intensity values,
the remaining component~$\tilde\feature\in L^2(\domain,\R^C)$ represents deep features, which are high-dimensional local image patterns describing the local structure of the image
as a superposition on different levels of a multiscale image approximation.
Let us denote by $\proj$ the projection onto the image component of a feature, i.e.~$\proj[\feature]=\image$.
To compute the geodesic sequence in the deep feature space, we extract the features~$\FeatureOperator(\image_A),\FeatureOperator(\image_B)\in L^2(\domain,\R^C)$
from the fixed input images $\image_A,\image_B\in \imageSpace$ and define for a fixed (small)~$\eta>0$
\[
\feature_A=(\eta\image_A,\FeatureOperator(\image_A))\,,\quad
\feature_B=(\eta\image_B,\FeatureOperator(\image_B))\,.
\]
The computation of the VGG features is composed of convolution operators and nonlinear ReLU activation functions which are both continuous mappings.
Hence, it is reasonable to assume in our mathematical model that the mapping $\FeatureOperator:\imageSpace\to L^2(\domain,\R^C)$ is \emph{continuous}.
Following~\cite{SiZi15}, we define for the fully discrete model discussed in section~\ref{sec:fullyDiscrete} a discrete feature operator to
incorporate semantic information in image morphing based on convolutional neural networks, where $C$ ranges from 64 to 512.
The parameter~$\eta$ is used to scale down the RGB component mainly needed to compute the anisotropy (see below)
and to primarily focus on the actual VGG features when estimating the transport.

Next, we include an anisotropic elliptic operator~$L$ in our model to properly account for image structures such as sharp edges or corners. To this end, we consider an \emph{anisotropy operator $\anisotropy:\imageSpace\to L^\infty(\domain)$} fulfilling the following assumptions:
\begin{enumerate}[leftmargin=5.6ex,itemsep=.2\baselineskip,parsep=.1\baselineskip,label=(A\arabic*)]
\item\label{a1}\emph{boundedness and coercivity}:
$c_\anisotropy<\anisotropy[\image](x)<C_\anisotropy$ for $0< c_\anisotropy<C_\anisotropy$ 
and all $\image\in \imageSpace$ and a.e.~$x\in\domain$,
\item\label{a2}\emph{compactness}:
$\image_k\rightharpoonup\image$ in $\imageSpace$ implies $\anisotropy[\image_k]\to\anisotropy[\image]$ in $L^\infty(\domain)$,
\item\label{a3}\emph{Lipschitz continuity}:
for all neighborhoods $\mathcal{U}\subset\imageSpace$ there exists $L_\anisotropy>0$ such that
$\Vert\anisotropy[\image]-\anisotropy[\tilde\image]\Vert_{L^\infty}\leq L_\anisotropy\Vert\image-\tilde\image\Vert_{\imageSpace}$ 
for all $\image,\tilde\image\in\mathcal{U}$.
\end{enumerate}
In the numerical experiments, we use the operator~\cite{PeMa90}
\begin{equation}
\anisotropy[\image](x)=\exp\left(-\frac{\Vert(\mathcal{G}_{\rho}\ast D\mathcal{G}_{\sigma}\ast\image)(x)\Vert_2^2}{\xi_1}\right)+\xi_2,
\label{eq:anisotropy_example}
\end{equation}
for fixed $\xi_1,\xi_2>0$, where $\mathcal{G}_{\sigma},\mathcal{G}_{\rho}$ are the Gaussian kernels with standard deviation~$\sigma,\rho>0$.
Note that \eqref{eq:anisotropy_example} satisfies \ref{a1}--\ref{a3}. In fact, the anisotropy operator~$a$ is a scale factor for the elliptic operator of the deformation field, 
which nearly vanishes in the proximity of interfacial structures.
Thus, large deformation gradients are less penalized in these regions and consequently sharp edges can be better 
preserved along geodesic paths.

Now we are in the position to introduce the variational model for  deep feature metamorphosis.
Instead of generalizing the definition of regular paths and adapting the notion of a weak material derivative \eqref{eq:weakz} 
originally proposed by Trouv{\'e} and Younes, we follow the relaxed material derivative approach proposed in \cite{EfNe19}, 
in which the material derivative quantity is retrieved from a variational inequality.
In~\cite[Section~3]{EfNe19}, the equivalence of this energy functional and~\eqref{eq:DefinitionPathenergyImage} in the isotropic case has been shown.
Let $\psi$ as above denote the Lagrangian flow map induced by the Eulerian motion field with $\dot\psi_t(x)=v(t,\psi_t(x))$ and $\psi_0(x)=x$.
Then, we replace the equality \eqref{eq:Lagrange} (rephrased for the feature map $\feature$ as $\feature(t,\psi_t)-\feature(s,\psi_s)=\int_s^t \tilde z(r,\psi_r) \dx r$ with $\tilde z\in L^2((0,1)\times\Omega,\R^{3+C})$ 
being the weak material derivative) by the inequality 
\begin{align}\label{eq:zineq}
|\feature(t,\psi_t(x))-\feature(s,\psi_s(x))|\leq\int_s^t z(r,\psi_r(x))\dx r
\end{align}
for a.e.~$x\in \Omega$ and all $1\geq t > s \geq 0$, where formally the scalar valued $z=|\tilde z|$ replaces the actually vector-valued material derivative.
In fact, this inequality defines a set $\mathcal{C}(\feature)$ of admissible pairs $(v,z)$ given a path $\feature$ in $L^2([0,1],\featureSpace)$.
This relaxed approach will turn out to be very natural when it comes to lower semicontinuity of the path energy in the context of the existence proof for geodesic paths.
For more details we refer the reader to Section~\ref{sec:Mosco}.

\begin{definition}[Continuous path energy]\label{contPathEnergy}
We consider the anisotropic elliptic operator
\[
L[\tilde\anisotropy,v,v]=\tilde\anisotropy\left(\frac{\lambda}{2}(\tr\varepsilon[v])^2+\mu\tr(\varepsilon[v]^2)\right)+\gamma|D^m v|^2
\]
for an anisotropy weight~$\tilde\anisotropy\in L^\infty(\domain)$, a velocity field~$v\in\motionSpace$ and $\gamma,\mu,\lambda>0$. 
Then, we define the \emph{path energy}
\begin{equation}
\pathenergy[\feature]=\int_0^1\inf_{(v,z)\in\mathcal{C}(\feature)}\int_\domain L[\anisotropy[\proj[\feature]],v,v]+\frac{1}{\delta}z^2\dx x\dx t
\label{eq:pathEnergyDeep}
\end{equation}
for a path $\feature\in L^2([0,1],\featureSpace)$, where 
\[
\mathcal{C}(\feature)\subset L^2((0,1),\motionSpace)\times L^2((0,1) \times \domain)
\] 
denotes the set of admissible pairs of the velocity and a scalar quantity $z$ fulfilling \eqref{eq:zineq}.
\end{definition}
Let us stress that the anisotropy $\tilde\anisotropy=\anisotropy[\proj[\feature]]$
solely takes into account local RGB values and not the actual VGG features with their discriminative multiscale characteristics.

Geodesic curves~$\feature\in L^2([0,1],\featureSpace)$ in the deep feature space joining $\feature_A,\feature_B\in\featureSpace$ are 
defined as minimizers of the path energy~$\pathenergy$ among all curves
with the fixed boundary conditions~$\feature(0)=\feature_A$ and $\feature(1)=\feature_B$.

\begin{remark}\label{rem:equi}
One observes that a path~$\feature\in L^2([0,1],\featureSpace)$ in feature space with finite energy $\pathenergy[\feature]<\infty$ exhibits additional smoothness properties.
Indeed, the boundedness of $v$ in $L^2((0,1),H^m(\domain,\R^n))$ implies that the flow is in $\psi\in H^{1}((0,1),H^m(\domain,\domain))$
and, by using Sobolev embedding arguments, in $C^{0,\frac{1}{2}}([0,1],C^{1,\alpha}(\overline\domain,\overline\domain))$ with $\alpha \in (0,\min\{1,m-1-\tfrac{n}2\})$.
The same observation holds for $\psi^{-1}$ by noting that $\psi^{-1}_t(\cdot)$ is the flow associated with the backward motion field $-v(1-t,\cdot)$.
This together with the variational inequality~\eqref{eq:zineq} and $z$ in $L^2((0,1)\times\domain)$ ensures
that $t \mapsto \feature(t,\psi(t,\cdot))\in H^1((0,1),\featureSpace)\subset C^{0,\frac12}([0,1],\featureSpace)$.
Using approximation by smooth functions one shows that $t\mapsto\feature(t,\cdot)\in C^{0}([0,1],\featureSpace)$ is uniformly continuous, and by using \ref{a3}
the mapping $t\mapsto\anisotropy[\proj[\feature(t,\cdot)]]$ is well-defined and in $C^0([0,1],L^\infty(\domain))$.
\end{remark}


\section{Variational time discretization}\label{sec:timeDiscrete}
In this section, we develop a variational time discretization of the deep feature space metamorphosis model taking into account the approach presented in~\cite{RuWi12b,BeEf14}.

We define the \emph{time discrete pairwise energy} for two feature maps $\feature,\tilde\feature\in \featureSpace$ by
\[
\energy[\feature,\tilde\feature]=\min_{\deformation\in\admset}\energy^D[\anisotropy[\proj[\tilde \feature]],\feature,\tilde\feature,\deformation]\,,
\]
where $\energy^D:L^\infty(\domain)\times \featureSpace\times \featureSpace\times\admset\to\R$ is given by
\begin{equation}
\energy^D[\tilde\anisotropy,\feature,\tilde{\feature},\deformation]
=\int_\domain\tilde\anisotropy\energyDensity(D\deformation)+\gamma|D^m \deformation|^2+\frac{1}{\delta}|\tilde{\feature}\circ\deformation-\feature|^2\dx x\,.
\label{eq:energy}
\end{equation}
Here, the \emph{set of admissible deformations} is
\begin{equation*}
\admset=\{\deformation\in H^m(\domain,\domain):\det(D\deformation)>0\text{ a.e. in }\domain,~
\deformation|_{\partial\Omega}=\Id\}.
\end{equation*}
Note that the anisotropy weight only depends on the image component of the second feature~$\tilde\feature$ in the pairwise energy.
We make the following assumptions with respect to the \emph{energy density function~$\energyDensity$}:
\begin{enumerate}[leftmargin=6.5ex,itemsep=.2\baselineskip,parsep=.1\baselineskip,label=(W\arabic*)]
\item\label{W1}
$\energyDensity:\R^{n,n}\to\R^+_0$ and $\energyDensity \in C^4(\GL^+(n))$ is polyconvex and
$\energyDensity(\Id)=0$, $D\energyDensity(\Id)=0$,
\item\label{W2}
there exist constants $C_{\energyDensity,1},C_{\energyDensity,2},r_\energyDensity>0$ such that for
all $A\in\GL^+(n)$ the growth estimates
\begin{align*}
\energyDensity(A)&\geq C_{\energyDensity,1}\vert A^\sym-\Id\vert^2\,,
&&\text{if }\vert A-\Id\vert<r_\energyDensity\,,\\
\energyDensity(A)&\geq C_{\energyDensity,2}\,,
&&\text{if }\vert A-\Id\vert \geq r_\energyDensity
\end{align*}
hold true,
\item\label{W3}
for all $A\in\R^{n,n}$ the relation
\[
\frac{1}{2}D^2\energyDensity(\Id)(A,A)=\frac{\lambda}{2}(\tr A)^2+\mu \tr((A^\sym)^2)
\]
holds true.
\end{enumerate}
The first two assumptions ensure existence of a minimizing deformation in $\eqref{eq:energy}$ and the third is a consistency assumption with respect to the differential operator $L$ required to guarantee that the below defined discrete path energy is consistent 
with the time continuous path energy \eqref{eq:pathEnergyDeep}.

The particular energy density function
\begin{equation}
\energyDensity(D\deformation)=\frac{\lambda}{2}\left(e^{(\log \det (D\deformation))^2}-1\right)+\mu|\varepsilon[\deformation]-\Id|^2
\label{eq:energyDensityExample}
\end{equation}
used for all numerical experiments satisfies \ref{W1}--\ref{W3}.
The first term enforces the positivity of the determinant of the Jacobian matrix of a deformation
and favors a balance of shrinkage and growth as advocated in \cite{DrRu04,BuMo13},
while the second term penalizes large deviations of the deformation from the identity.
Here, the positivity constraint of the determinant of the Jacobian of the deformations prohibits interpenetration of matter~\cite{Ba81}.

We proceed with the definition of the discrete path energy and the discrete geodesic between two features $\feature_A=(\eta\image_A,\FeatureOperator(\image_A)),\feature_B=(\eta\image_B,\FeatureOperator(\image_B))\in \featureSpace$.
\begin{definition}[Discrete path energy]\label{def:definition_of_geodesics}
Let $K\geq 1$ and $\feature_0=\feature_A,\feature_K=\feature_B\in\featureSpace$.
The \emph{discrete path energy~$\Pathenergy^K$} for a discrete $(K+1)$-path $\featurevec=(\feature_0,\dots,\feature_K)\in \featureSpace^{K+1}$ is defined as
\begin{equation}
\Pathenergy^K[\featurevec]\coloneqq K\sum_{k=1}^K\energy[\feature_{k-1},\feature_k]\,.
\label{eq:pathenergy}
\end{equation}
A \emph{discrete geodesic path} morphing $\feature_A \in \featureSpace$ into $\feature_B \in \featureSpace$ is a discrete $(K+1)$-tuple that minimizes $\Pathenergy^K$
over all discrete paths $\featurevec=(\feature_A,\featureinnervec,\feature_B)\in \featureSpace^{K+1}$ with $\featureinnervec=(\feature_1,\ldots,\feature_{K-1})\in \featureSpace^{K-1}$.
\end{definition}
For arbitrary vectors $\featurevec=(\feature_0,\ldots,\feature_K)\in\featureSpace^{K+1}$ and $\defvec=(\deformation_1,\ldots,\deformation_K)\in\admset^K$ we set
\begin{equation}
\Pathenergy^{K,D}[\featurevec,\defvec]\coloneqq K\sum_{k=1}^{K}\energy^D[\anisotropy[\proj[\feature_k]],\feature_{k-1},\feature_k,\deformation_k]\,.
\label{eq:definitionPathenergyD}
\end{equation}

In what follows, we will investigate the existence of discrete geodesic curves in the time discrete deep feature space metamorphosis model.
To this end, we combine the proofs of the local well-posedness of the pairwise energy~$\energy$
with the existence result of a feature vector minimizing~$\Pathenergy^{K,D}$ for a fixed vector of deformations.
We remark that the structure of all proofs is similar to the corresponding proofs in~\cite{BeEf14,Ef18} and we focus on 
the adaptations necessitated by the anisotropic regularization.

The following lemma, which provides an estimate for the $H^m(\domain)$-norm of the displacement, is crucial for the well-posedness of the energy.
\begin{lemma}\label{lemm:growthControl}
Let \ref{W1}--\ref{W2} and \ref{a1} be satisfied.
Then there exists a continuous and monotonically increasing function $\theta:\R^+_0\to\R^+_0$ with $\theta(0)=0$, which only depends on
$\domain$, $m$, $n$, $\gamma$, $c_\anisotropy$, $C_{\energyDensity,1}$, $C_{\energyDensity,2}$ and $r_{\energyDensity}$, such that
\begin{equation*}
\Vert\deformation-\Id\Vert_{H^m(\domain)}\leq\theta\left(\energy^D[\anisotropy[\proj[\tilde\feature]],\feature,\tilde\feature,\deformation]\right)
\end{equation*}
for all $\feature, \tilde\feature \in \featureSpace$ and all $\deformation \in \admset$. Furthermore,
$\theta(x)\leq C(x+x^2)^{\frac12}$ for a constant $C>0$.
\end{lemma}
\begin{proof}
Set $\overline\energy=\energy^D[\anisotropy[\proj[\tilde\feature]], \feature, \tilde\feature, \deformation]$.
An application of the Gagliardo--Nirenberg inequality~\cite{Ni66} yields
\begin{equation}
\Vert\deformation-\Id\Vert_{H^m(\domain)}\leq C(\Vert\deformation-\Id\Vert_{L^2(\domain)}+|\deformation-\Id|_{H^m(\domain)})\,.
\label{eq:mGrowthDisplacement}
\end{equation}
The last term in \eqref{eq:mGrowthDisplacement} is bounded by
\begin{equation}
|\deformation-\Id|_{H^m(\domain)}=|\deformation|_{H^m(\domain)}\leq\sqrt{\tfrac{\overline{\energy}}{\gamma}}\,.
\label{eq:displacementHigherOrderControl}
\end{equation}
By using the embedding of $H^m(\domain,\domain)$ into $C^{1,\alpha}(\overline\domain,\overline\domain)$ and the uniform boundedness of the minimizing sequence in $L^2(\domain,\domain)$ we get 
$\Vert\deformation-\Id\Vert_{C^{1,\alpha}(\overline\domain)}
\leq C+C\sqrt{\overline{\energy}}\,.$
To control the lower order term appearing on the right-hand side of~\eqref{eq:mGrowthDisplacement}, we define
$\mathcal{S}=\{x\in\domain:|D\deformation(x)-\Id|<r_\energyDensity\}$ and use \ref{a1} and \ref{W2} to obtain
\[
|\domain\backslash\mathcal{S}|c_\anisotropy C_{\energyDensity,2}\leq\int_\domain\anisotropy[\proj[\tilde\feature]]\energyDensity(D\deformation) \dx x\leq\overline{\energy}\,,
\]
which implies $|\domain\backslash\mathcal{S}|\leq\frac{\overline{\energy}}{c_\anisotropy C_{\energyDensity,2}}$.
Hence, by the embedding $H^m(\domain,\domain)\hookrightarrow C^1(\overline\domain,\overline\domain)$ 
we infer
\begin{align}
&\int_\domain|\varepsilon[\deformation]-\Id|^2 \dx x\notag\\
\leq&\int_\mathcal{S}\frac{\energyDensity(D\deformation)}{C_{\energyDensity,1}}\dx x +|\domain\backslash\mathcal{S}|\left(C+C\sqrt{\overline{\energy}}\right)^2 \notag\\
\leq&\frac{\overline{\energy}}{C_{\energyDensity,1}}+\frac{\overline{\energy}}{c_\anisotropy C_{\energyDensity,2}}\left(C+C\overline{\energy}\right)\,.
\label{eq:convergenceLowerLTwo} 
\end{align}
We remark that the inequality
\begin{equation}
\Vert\deformation-\Id\Vert_{L^2(\domain)}\leq C\Vert\varepsilon[\deformation]-\Id\Vert_{L^2(\domain)}
\label{eq:KornEstimate}
\end{equation}
holds true, which follows from Korn's inequality and the Poincar\'e inequality.
Thus, the lemma follows by combining \eqref{eq:mGrowthDisplacement}, \eqref{eq:displacementHigherOrderControl}, \eqref{eq:convergenceLowerLTwo} and \eqref{eq:KornEstimate}.
\end{proof}
\begin{proposition}[Well-posedness of $\energy$]\label{prop:wellPosednessEnergy}
Let $\feature\in\featureSpace$ be a fixed feature vector.
Under the assumptions \ref{W1}--\ref{W2} and \ref{a1}, there exists a constant~$C_\energy$ (depending on
$\domain,m,n,\gamma,\delta,\mu,\lambda,c_\anisotropy,C_{\energyDensity,1},C_{\energyDensity,2},r_{\energyDensity}$) such that for every fixed
\begin{equation}
\tilde\feature\in\left\{g\in \featureSpace:\Vert\feature-g\Vert_{\featureSpace}<C_\energy\right\}
\label{eq:closenessFeature}
\end{equation}
there exists $\deformation\in\admset$ which minimizes $\energy^D[\anisotropy[\proj[\tilde\feature]],\feature,\tilde\feature,\cdot]$
defined in \eqref{eq:energy} and $\deformation$ is a $C^1(\domain,\domain)$-diffeo\-morphism.
\end{proposition}
\begin{proof}
For fixed $\feature\in\featureSpace$, let $\tilde\feature$ be a feature vector satisfying~\eqref{eq:closenessFeature} for a constant~$C_\energy$ specified below.
Let $\{\deformation^j\}_{j\in\N}\in\admset$ be any sequence such that the mismatch $\energy^D[\anisotropy[\proj[\tilde\feature]],\feature,\tilde\feature,\deformation^j]$ converges to 
$\underline{\mathbf{W}}=\inf_{\deformation\in\admset}\energy^D[\anisotropy[\proj[\tilde\feature]],\feature,\tilde\feature,\deformation]\geq 0.$
Since $\Id\in\admset$ we can deduce using~\ref{W1} that
\begin{equation*}
\underline{\mathbf{W}}\leq\energy^D[\anisotropy[\proj[\tilde\feature]],\feature,\tilde\feature,\deformation^j]
\leq\overline{\mathbf{W}}\coloneqq\energy^D[\anisotropy[\proj[\tilde\feature]],\feature,\tilde\feature,\Id]=\tfrac{1}{\delta}\Vert\tilde\feature-\feature\Vert_{\featureSpace}^2<\frac{C_\energy^2}{\delta}
\end{equation*}
for all $j\in\N$.
Using again the Gagliardo--Nirenberg inequality we infer that
$\{\deformation^j\}_{j\in\N}$ is uniformly bounded in~$H^m(\domain,\domain)$ because of
the estimate $|\deformation^j|_{H^m(\domain)}^2\leq\frac{\overline{\mathbf{W}}}{\gamma}$. 
Due to the reflexivity of $H^m(\domain,\domain)$ there exists a weakly convergent subsequence (not relabeled) such that $\deformation^j\rightharpoonup\deformation$ in $H^m(\domain,\domain)$.
By using the Sobolev embedding theorem as well as the Arzel\`a--Ascoli theorem we can additionally infer that for a subsequence (again not relabeled)
$\deformation^j\to\deformation$ in $C^{1,\alpha}(\overline\domain,\overline\domain)$ for $\alpha\in(0,m-1-\frac{n}{2})$ holds true.
Then, Lemma~\ref{lemm:growthControl} implies
\[
\Vert\deformation^j-\Id\Vert_{C^1(\overline\domain)}\leq C\theta(\overline{\mathbf{W}})<C\theta(\delta^{-1}C_{\energy}^2)\,.
\]
Thus, by choosing $C_\energy$ sufficiently small and taking into account the Lipschitz continuity of the determinant we obtain
$
\Vert\det(D\deformation^j)-1\Vert_{L^\infty(\domain)}\leq C_{\det}
$
for a constant $C_{\det}\in(0,1)$ and all $j\in\N$, which implies $\det(D\deformation^j)\geq C>0$ for a constant $C$.
Note that all estimates remain valid for the limit deformation~$\deformation\in\admset$.
By~\cite[Theorem 5.5-2]{Ci88} the deformations~$\{\deformation^j\}_{j\in\N}$ and $\deformation$ are $C^1(\domain,\domain)$-diffeomorphisms.
Finally, \ref{W1} and the lower semicontinuity of the seminorm imply
\begin{equation*}
\liminf_{j\to\infty}\int_\domain\anisotropy[\proj[\tilde\feature]]\energyDensity(D\deformation^j)+\gamma|D^m\deformation^j|^2\dx x
\geq\int_\domain\anisotropy[\proj[\tilde\feature]]\energyDensity(D\deformation)+\gamma|D^m\deformation|^2\dx x\,.
\end{equation*}

It remains to verify that
\begin{equation}
\Vert\tilde\feature\circ\deformation^j-\feature\Vert_{\featureSpace}\to\Vert\tilde\feature\circ\deformation-\feature\Vert_{\featureSpace}
\label{eq:mismatchConvergence}
\end{equation}
as $j\to\infty$.
To this end, we approximate $\tilde\feature$ by smooth functions $\tilde\feature^i\in C^\infty(\domain,\R^{3+C})$ with
$\Vert\tilde\feature-\tilde\feature^i\Vert_{\featureSpace}\to 0$.
Then, using the transformation formula we obtain
\begin{align*}
&\Vert\tilde\feature\circ\deformation^j-\tilde\feature\circ\deformation\Vert_{\featureSpace}\\
\leq&
\Vert\tilde\feature\circ\deformation^j-\tilde\feature^i\circ\deformation^j\Vert_{\featureSpace}+\Vert\tilde\feature^i\circ\deformation^j-\tilde\feature^i\circ\deformation\Vert_{\featureSpace}
+\Vert\tilde\feature^i\circ\deformation-\tilde\feature\circ\deformation\Vert_{\featureSpace}\\
\leq&
\Vert\tilde\feature-\tilde\feature^i\Vert_{\featureSpace}\Big(\Vert\det(D(\deformation^j)^{-1})\Vert_{L^\infty(\domain)}^{\frac12}
+\Vert\det(D\deformation^{-1})\Vert_{L^\infty(\domain)}^{\frac12}\Big)\\
&+\Vert D\tilde\feature_i\Vert_{L^\infty(\domain)}\Vert\deformation^j-\deformation\Vert_{L^2(\domain)}\,,
\end{align*}
where $\det(D(\deformation^j))^{-1}$ and $\det(D(\deformation))^{-1}$ are pointwise estimated by $(1-C_{\det})^{\frac12}\,.$
Finally, by first choosing~$i$ and then~$j$ we obtain~\eqref{eq:mismatchConvergence} and thereby verify the claim.
\end{proof}

This proposition guarantees the existence of an admissible vector of deformations~$\defvec\in\admset^K$
such that $\Pathenergy^{K,D}[\featurevec,\defvec]=\Pathenergy^K[\featurevec]$
provided that each pair of features $(\feature_k,\feature_{k+1})$ contained in
$\featurevec=(\feature_0,\ldots,\feature_K)\in\featureSpace^{K+1}$ satisfies~\eqref{eq:closenessFeature}.

In what follows, we prove the existence of an energy minimizing vector of features for a fixed vector of deformations.
\begin{proposition}\label{prop:existenceImageInnerVec}
Let $K\geq 2$, $\feature_A,\feature_B\in\featureSpace$ and $\defvec=(\deformation_1,\ldots,\deformation_K)\in\admset^K$ be fixed.
We assume that the deformations satisfy
\begin{equation}
\min_{k\in\{1,\ldots,K\}}\min_{x\in\overline\domain}\det(D\deformation_k(x))\geq c_{\det}
\label{eq:uniformDetBoundMinimizer}
\end{equation}
for a constant $c_{\det}>0$.
Then, under the assumptions \ref{W1}--\ref{W2} and \ref{a1}--\ref{a2} there exists a feature vector $\featurevec$
with $\feature_0=\feature_A$ and $\feature_K=\feature_B$ such that
\begin{equation*}
\Pathenergy^{K,D}[\featurevec,\defvec]=\inf\left\{\Pathenergy^{K,D}[(\feature_A,\featureinnervectest,\feature_B),\defvec]:\featureinnervectest\in \featureSpace^{K-1}\right\}\,.
\end{equation*}
\end{proposition}
\begin{proof}
We consider a minimizing sequence of features $\featureinnervec^j=(\feature^j_1,\ldots,\feature^j_{K-1})\in\featureSpace^{K-1}$, $j\in\N$, for the energy
$\featureinnervectest \mapsto \Pathenergy^{K,D}[(\feature_A,\featureinnervectest,\feature_B),\defvec]$. Then,
\begin{equation*}
0\leq\Pathenergy^{K,D}[(\feature_A,\featureinnervec^j,\feature_B),\defvec]
\leq\Pathenergy^{K,D}[(\feature_A,(\feature_A,\ldots,\feature_A),\feature_B),\defvec]\eqqcolon\overline{\Pathenergy^{K,D}}\,.
\end{equation*}
A straightforward computation reveals
\begin{align*}
\overline{\Pathenergy^{K,D}}
\leq&K\sum_{k=1}^K C_\anisotropy\Vert\energyDensity(D\deformation_k)\Vert_{L^1(\domain)}+\gamma \Vert\deformation_k\Vert_{H^m(\domain)}^2\\
&+\frac{CK^2}{\delta}\left((1+c_{\det}^{-1})\Vert\feature_A\Vert_{\featureSpace}^2+c_{\det}^{-1}\Vert\feature_B\Vert_{\featureSpace}^2\right)\,,
\end{align*}
where we used \ref{a1},~\eqref{eq:uniformDetBoundMinimizer} and the transformation formula. 
Furthermore, again by~\eqref{eq:uniformDetBoundMinimizer} one obtains
\begin{equation}
\Vert\feature_k^j\Vert_{\featureSpace}\leq\Vert\feature_{k+1}^j\circ\deformation_{k+1}-\feature_k^j\Vert_{\featureSpace}+\Vert\feature_{k+1}^j\circ\deformation_{k+1}\Vert_{\featureSpace}
\leq\sqrt{\tfrac{\delta\overline{\Pathenergy^{K,D}}}{K}}+c_{\det}^{-\frac{1}{2}}\Vert\feature_{k+1}^j\Vert_{\featureSpace}\,.
\label{eq:uniformBoundFeaturesInduction}
\end{equation}
Thus, an induction argument (starting from $k=K-1$) shows that $\featureinnervec^j=(\feature^j_1,\ldots,\feature^j_{K-1})$ is uniformly bounded in $\featureSpace^{K-1}$ independently of $j$, which
implies for a subsequence (not relabeled) $\featureinnervec^j\rightharpoonup\featureinnervec$ in $\featureSpace^{K-1}$.

In what follows, we prove the weak lower semicontinuity of the discrete path energy along the minimizing sequence.
We observe that~\ref{a2} implies $\anisotropy[\proj[\feature^j_k]]\to\anisotropy[\proj[\feature_k]]$ in $L^\infty(\domain)$, which yields
\[
\lim_{j\to\infty}\int_\domain\anisotropy[\proj[\feature^j_k]]\energyDensity(D\deformation_k)\dx x=\int_\domain\anisotropy[\proj[\feature_k]]\energyDensity(D\deformation_k)\dx x
\]
for every $k=1,\ldots,K$.
It remains to verify the weak lower semicontinuity of the matching functional, i.e.
\begin{equation}
\Vert\feature_k\circ\deformation_k-\feature_{k-1}\Vert_{\featureSpace}^2\leq\liminf_{j\to\infty}\Vert\feature^j_k\circ\deformation_k-\feature^j_{k-1}\Vert_\featureSpace^2
\label{eq:lowerSemicontinuityMatching}
\end{equation}
for every $k=1,\dots,K$.
To this end, we first show $\feature_k^j\circ\deformation_k\rightharpoonup\feature_k\circ\deformation_k$ in~$\featureSpace$.
For every $g \in \featureSpace$ the transformation formula yields
\begin{equation*}
\int_{\domain}(\feature^j_k\circ\deformation_k-\feature_k\circ\deformation_k)\cdot g\dx x
=\int_{\domain}(\feature^j_k-\feature_k)\cdot(g(\det (D\deformation_k))^{-1})\circ\deformation_k^{-1}\dx x\,,
\end{equation*}
which converges to $0$ since $(g(\det (D\deformation_k))^{-1})\circ\deformation_k^{-1}\in\featureSpace$ due to~\eqref{eq:uniformDetBoundMinimizer}.
Hence, $\feature_k^j\circ\deformation_k-\feature_{k-1}^j\rightharpoonup\feature_k\circ\deformation_k-\feature_{k-1}$ in~$\featureSpace$,
which readily implies~\eqref{eq:lowerSemicontinuityMatching}.
Hence,
\begin{align*}
\liminf_{j\to\infty}\Pathenergy^{K,D}[(\feature_A,\featureinnervec^j,\feature_B),\defvec]\geq\Pathenergy^{K,D}[(\feature_A,\featureinnervec,\feature_B),\defvec]\,,
\end{align*}
which proves the proposition.
\end{proof}
We can now combine both previous propositions to prove the existence of discrete geodesics for the deep feature space metamorphosis model.
\begin{theorem}[Existence of discrete geodesics]\label{thm:metaExistenceGeodesics}
Let the assumptions \ref{W1}--\ref{W2} and \ref{a1}--\ref{a2} be satisfied, $K\geq 2$ and $\feature_A\in\featureSpace$.
Then, there exists a constant $C_{\Pathenergy}>0$, which is independent of $K$, such that for every
\begin{equation}\label{eq:proximityFeatures}
\feature_B\in\left\{g\in\featureSpace:\Vert g-\feature_A\Vert_\featureSpace<C_{\Pathenergy}\sqrt{K}\right\}
\end{equation}
there exists $\featureinnervec\in\featureSpace^{K-1}$ such that
\begin{equation*}
\Pathenergy^K[(\feature_A,\featureinnervec,\feature_B)]=\inf_{\featureinnervectest\in\featureSpace^{K-1}}\Pathenergy^K[(\feature_A,\featureinnervectest,\feature_B)]
\end{equation*}
and the associated vector of minimizing deformations consists of $C^1(\domain,\domain)$-diffeo\-morphisms.
\end{theorem}
\begin{proof}
For a fixed $\feature_A\in\featureSpace$ let $\feature_B$ satisfy \eqref{eq:proximityFeatures} for a constant $C_{\Pathenergy}$ specified below.
For $k=0,\ldots,K$ let $\overline\feature_k=\frac{k}{K}\feature_B+(1-\frac{k}{K})\feature_A\in\featureSpace$ be a convex combination of the input features.
We first note that 
\begin{align*}
\overline{\Pathenergy^K}\coloneqq&\Pathenergy^{K,D}[(\overline\feature_0,\overline\feature_1,\ldots,\overline\feature_K),(\Id,\ldots,\Id)]\\
=&\frac{K}{\delta}\sum_{k=1}^K \Vert\feature_{k}-\feature_{k-1}\Vert^2_{\featureSpace}=\frac{1}{\delta}\Vert\feature_B-\feature_A\Vert_{\featureSpace}^2<\frac{C_{\Pathenergy}^2 K}{\delta}
\end{align*}
is a finite upper bound for the energy.
Consider the minimizing sequence
\[
(\featurevec^j,\defvec^j)=
((\feature_0^j,\ldots,\feature_K^j),(\deformation_1^j,\ldots,\deformation_K^j))\in\featureSpace^{K+1}\times\admset^K
\]
for $j\in\N$ with $\feature_0^j=\feature_A$ and $\feature_K^j=\feature_B$
associated with the variational problem
$(\featurevec,\defvec)\mapsto\Pathenergy^{K,D}[\featurevec,\defvec]$,
which has the finite upper bound $\overline{\Pathenergy^K}$.
Following the same line of arguments as in Proposition~\ref{prop:wellPosednessEnergy} we obtain the boundedness of $\defvec^j$ in~$H^m(\domain,\domain)$, 
which results in a weakly convergent subsequence (not relabeled) $\defvec^j\rightharpoonup\defvec$ in $H^m(\domain,\domain)$.
Due to $H^m(\domain,\domain)\hookrightarrow C^1(\overline\domain,\overline\domain)$ one obtains $\defvec^j\to\defvec$ in~$C^1(\overline\domain,\overline\domain)$ for a further subsequence (not relabeled).
By taking into account Lemma~\ref{lemm:growthControl} we get
\begin{equation*}
\Vert\deformation^j_k-\Id\Vert_{C^1(\overline\domain)}\leq C\Vert\deformation^j_k-\Id\Vert_{H^m(\domain)}
\leq C\theta(K^{-1}\overline{\Pathenergy^K})\leq C\theta(\delta^{-1}C_{\Pathenergy}^2)
\end{equation*}
for every $j\in\N$ and every $k=1,\dots,K$.
By adapting~$C_{\Pathenergy}$ if necessary we can assume
\[
\inf_{j\in\N}\min_{k=1,\ldots,K}\min_{x\in\overline\domain} \;\det(D\deformation^j_k(x))>c_{\det}
\]
for a constant $c_{\det}>0$.
Taking into account \cite[Theorem 5.5-2]{Ci88} we can conclude that $\defvec^j$ and $\defvec$ are $C^1(\domain,\domain)$-diffeomorphisms.
Using Proposition~\ref{prop:existenceImageInnerVec} we can replace~$\featurevec^j$ by the energy minimizing feature vector associated with $\defvec^j$, which possibly reduces the path energy. 
The features $\featurevec^j$ are uniformly bounded in~$\featureSpace^{K+1}$, which follows from an analogous reasoning as~\eqref{eq:uniformBoundFeaturesInduction}.
Thus, $\featurevec^j\rightharpoonup\featurevec$ holds true for a subsequence (not relabeled) in $\featureSpace^{K+1}$, which
implies $\anisotropy[\proj[\feature^j_k]] \to \anisotropy[\proj[\feature_k]]$ in $L^\infty(\domain)$ due to~\ref{a2}. Consequently, for every $k=1,\dots,K$ we obtain
\begin{equation*}
\liminf_{j\to\infty}\int_{\domain}\!{\anisotropy[\proj[\feature^j_k]]\energyDensity(D\deformation^j_k)}\dx x
\geq \!\!\int_{\domain}\!{\anisotropy[\proj[\feature_k]]\energyDensity(D\deformation_k )}\dx x\,.
\end{equation*}

Finally, we verify the lower semicontinuity estimate 
\begin{equation}
\Vert\feature_k\circ\deformation_k-\feature_{k-1}\Vert_{\featureSpace}^2\leq\liminf_{j\to\infty}\Vert\feature^j_k\circ\deformation^j_k-\feature^j_{k-1}\Vert_{\featureSpace}^2
\label{eq:lowerSemicontinuityJointMatching}
\end{equation}
for every $k=1,\dots,K$. 
To this end, we take into account the decomposition  
\[
\feature^j_k\circ\deformation^j_k-\feature_k\circ\deformation_k = (\feature^j_k\circ\deformation^j_k-\feature_k\circ\deformation^j_k)+(\feature_k\circ\deformation^j_k-\feature_k\circ\deformation_k)\,.
\]
The second term is estimated as in the proof of \eqref{eq:mismatchConvergence}.
Thus it remains to consider the convergence properties of the first term.
For a test function $g\in\featureSpace$ we obtain using the transformation rule 
\begin{equation*}
\int_\domain(\feature^j_k\circ\deformation^j_k-\feature_k\circ\deformation^j_k)\cdot g\dx x
=\int_\domain(\feature^j_k-\feature_k)\cdot(g(\det(D\deformation^j_k))^{-1})\circ(\deformation^j_k)^{-1}\dx x\,.
\end{equation*}
The right hand side converges to~$0$ due to the convergence $(\det(D\deformation^j_k))^{-1}\circ(\deformation^j_k)^{-1}\to \det(D\deformation_k))^{-1}\circ\deformation_k^{-1}$ in $L^\infty(\Omega)$ and  $\feature_k^j \rightharpoonup \feature_k$ in $\featureSpace$ for $j\to \infty$.
Thus, $\feature_k^j\circ\deformation_k^j\rightharpoonup \feature_k\circ\deformation_k$ for $j\to \infty$, which together with the lower semicontinuity of the $L^2$-norm proves~\eqref{eq:lowerSemicontinuityJointMatching}. 
Altogether, we observe that 
$\Pathenergy^K[\featurevec]\leq\Pathenergy^{K,D}[\featurevec,\defvec]\leq\liminf\limits_{j\to\infty}\Pathenergy^{K,D}[\featurevec^j,\defvec^j]\,.$
\end{proof}

\section{Convergence of discrete geodesic paths}\label{sec:Mosco}
In this section, we provide a precise statement of the Mosco--convergence for $K\to\infty$ of a suitable temporal extension of the time discrete path energy~$\Pathenergy^K$
in the deep metamorphosis model to the time continuous path energy~$\pathenergy$ introduced in Definition~\ref{contPathEnergy}.
Furthermore, the convergence of time discrete geodesics to time continuous geodesic paths is established, which 
in particular implies the existence of time continuous geodesics in the deep feature metamorphosis model with an anisotropic regularizer.

We recall the definition of Mosco--convergence~\cite{Mosco69}, which can be seen as a modification of $\Gamma$--convergence. 
For further details we refer the reader to~\cite{dalMaso93}.
\begin{definition}\label{MoscoDefinition}
Let $X$ be a Banach space.
Consider functionals $\{\pathenergy^K\}_{K \in \N}$ and $\pathenergy$ from $X$ to $\overline{\R}$ that satisfy
\begin{itemize}[leftmargin=4.2ex,itemsep=.2\baselineskip,parsep=.1\baselineskip]
\item[(i)]
for every sequence $\{x^K\}_{K\in\N}\subset X$ with $x^K\rightharpoonup x \in X$ the estimate
\[\liminf_{K\to\infty}\pathenergy^K[x^K]\geq\pathenergy[x]\]
holds true ("$\liminf$--inequality"),
\item[(ii)]
for every $x \in X$ there exists a recovery sequence $\{x^K\}_{K\in\N}\subset X$ satisfying $x^K\rightarrow x$ in~$X$ such that the estimate
\[\limsup_{K\to\infty}\pathenergy^K[x^K]\leq\pathenergy[x]\] is valid.
\end{itemize}
Then $\{\pathenergy^K\}_{K \in \N}$ converges to~$\pathenergy$ in the sense of Mosco.
\end{definition}

In what follows, we define temporal extensions of all relevant quantities required for the statement of the Mosco--convergence.
We remark that this construction is similar to \cite{BeEf14,EfNe19}, where further details can be found.

To ensure that the involved deformations are diffeomorphisms and to
avoid the interpenetration of matter along the morphing sequence, we replace in the definition of 
$\admset$ the  positivity constraint for the determinant by  the stronger condition
$\det(D\deformation)\geq\varepsilon$ for a fixed (small)~$\varepsilon>0$ as already proposed in~\cite{EfNe19}.
Note that all existence results from the previous section remain valid for this modified definition.

For $K\in\N$, let $\defvec^K=(\deformation_1^K,\ldots,\deformation_K^K)\in\admset^K$ be the vector of deformations
associated with a vector of features $\featurevec^K=(\feature_0^K,\ldots,\feature_K^K)\in \featureSpace^{K+1}$. In particular, if $\defvec^K$ is optimal then $\Pathenergy^K[\featurevec^K]=\Pathenergy^{K,D}[\featurevec^K,\defvec^K]$.
For $k=1,\ldots,K$, we define for $t_k^K\coloneqq\frac{k}{K}$, $t\in[t_{k-1}^K,t_k^K]$, and $x\in\domain$
the \emph{discrete transport map} 
\begin{equation}
y_k^K(t,x)=x+(t-t^K_{k-1})K(\deformation_k^K(x)-x)\,.
\label{eq:transportMap}
\end{equation}
Note that $y_k^K(t^K_{k-1},x)=x$ and $y_k^K(t_k^K,x)=\deformation_k^K(x)$.
Following \cite[Chapter~5]{Ci88}, the condition 
\[
\max_{k=1,\ldots,K}\Vert D\deformation_k^K-\Id\Vert_{C^0(\overline\domain)}<1
\]
implies that $y_k^K(t,\cdot)$ is invertible, which follows for~$K$ large enough from the $\liminf$-part of the proof of Theorem~\ref{thm:gamma} below, and we denote the inverse by~$x_k^K(t,\cdot)$.
In this case, we consider the \emph{feature extension operator} $\interpolfeature^K[\featurevec^K,\defvec^K]\in L^2([0,1]\times\featureSpace)$ for $t\in[t^K_{k-1},t_k^K]$ by
\begin{equation*}
\interpolfeature^K[\featurevec^K,\defvec^K](t,x)
\coloneqq \Big(\feature_{k-1}^K+K(t-t^K_{k-1})(\feature_k^K\circ\deformation_k^K-\feature_{k-1}^K)\Big)(x_k^K(t,x))\,,
\end{equation*}
to define an extension $\pathenergy^K:L^2([0,1]\times\featureSpace)\to[0,\infty]$ 
of the discrete path energy~$\Pathenergy^{K,D}$, where 
\begin{align*}
&\pathenergy^{K}[\feature]= \inf_{\overline{\defvec}^K\in\admset^K}
\left\{\Pathenergy^{K,D}[\featurevec^K,\overline{\defvec}^K]:
\interpolfeature_K[\featurevec^K,\overline{\defvec}^K]=\feature\right\}
\end{align*}
if there exist $\featurevec^K\in\featureSpace^{K+1}$ and $\defvec^K\in\admset^K$
such that $\feature=\interpolfeature^K(\featurevec^K,\defvec^K)$,
else we set $\pathenergy^{K}[\feature]=\infty$.

We can now state the main theorem of this section:
\begin{theorem}[Mosco--convergence of the discrete path energies]\label{thm:gamma}
Let \ref{W1}--\ref{W3} and \ref{a1}--\ref{a3} be satisfied.
Then, the time discrete path energy~$\{\pathenergy^K\}_{K\in\N}$ converges to~$\pathenergy$ in the sense of Mosco in the $L^2([0,1]\times\featureSpace)$-topology for $K\to\infty$.
\end{theorem}
\begin{proof}
The proof follows the structure of the Mosco--convergence proof ~\cite[Theorem~5.2 \& Theorem~5.4]{EfNe19} 
with adaptations required due to the incorporation of the anisotropy. Furthermore, in our case images are not pointwise maps into a general Hadamard manifold but rather maps into some Euclidean space.
To keep the exposition compact, we focus here on these adaptations.
To facilitate reading, we give an overview of the general structure of the proof, which retrieves the overview of the proof structure in~\cite{EfNe19}.
Many of the technical arguments already appeared in the proof of existence of discrete geodesics in Section~\ref{sec:timeDiscrete} and were given there in full detail.
Thus, we keep these arguments brief here. 

\medskip

\noindent
(i) {\bf $\liminf$--inequality.}
\smallskip

\noindent
-- The \emph{identification of the image (feature) and deformation vectors} is unaltered compared to the proof in~\cite{EfNe19}.
Indeed, one obtains that the sequence $\feature^K$ of feature maps with uniformly bounded energy converges weakly to a feature map~$\feature\in L^2([0,1],\featureSpace)$ with finite energy.
In fact, $\featurevec^K$ and the optimal~$\overline{\defvec}^K$ can be retrieved from~$\feature^K$, where the existence of~$\overline{\defvec}^K$ 
follows as in~\cite[Lemma~5.1]{EfNe19}.

\medskip

\noindent
-- The verification of the \emph{lower semi-continuity of the weak material derivative} in the sense
\[
\Vert z\Vert_{L^2([0,1]\times\domain)}^2
\leq\liminf_{K\to\infty}K\sum_{k=1}^K\Vert\feature_{k-1}^K-\feature_k^K\circ\overline{\deformation}_k^K\Vert_{L^2(\domain)}^2\,.
\]
for $z$ being the weak limit of $z^K$ in $L^2((0,1) \times \domain)$ with
\[
z^K\big\vert_{[t_{k-1}^K,t_k^K)}\coloneqq K|\feature_{k-1}^K(x_k^K)-\feature_k^K\circ\overline{\deformation}_k^K(x_k^K)|
\]
is identical to the corresponding reasoning in~\cite{EfNe19}.
At this point we also observe that the velocity field~$w^K$ with $w^K\coloneqq w_k^K\coloneqq K(\overline{\deformation}_k^K-\Id)$ on $[t_{k-1}^K,t_k^K)$ is uniformly bounded in $L^2((0,1),\motionSpace)$ and thus
converges weakly in $L^2((0,1),\motionSpace)$ to some limit velocity field~$v$.
This fact is again proved following the corresponding reasoning as in \cite{EfNe19}.

\medskip

\noindent
-- Also the \emph{verification of the admissibility of the limit}, i.e.~$(v,z)\in \mathcal{C}(\feature)$,
stays unaltered compared to~\cite{EfNe19}. 
To this end, one shows that the discrete flow~$\psi^K$ associated with the motion field $v^K(t,x)\coloneqq K(\overline{\deformation}_k^K-\Id)(x_k^K(t,x))$ for $t\in[t_{k-1}^K,t_k^K)$ uniformly converges
in $C^{0,\alpha}([0,1],C^{1,\alpha}(\overline\domain))$ for a suitable constant~$\alpha>0$ to the continuous flow induced by the velocity~$v$.
Moreover, the variational inequality for the material derivative holds true in the limit.

\medskip

\noindent
-- In the final step, the \emph{lower semi-continuity of the viscous dissipation} has to be shown, i.e.
\begin{equation*}
\int_0^1\int_\domain L[\anisotropy[\proj[\feature]],v,v]\dx x\dx t
\leq \liminf_{K\to\infty}K\sum_{k=1}^K\int_\domain\anisotropy[\proj[\feature_k^K]]\energyDensity(D\overline\deformation_k^K)+\gamma\vert D^m\overline\deformation_k^K\vert^2\dx x\,.
\end{equation*}
Therefore, we define $\anisotropy^K,\overline{\anisotropy}^K\in L^\infty((0,1)\times\domain,\R^+)$ via
$\overline{\anisotropy}^K\big\vert_{[t_{k-1}^K,t_k^K)}\coloneqq\anisotropy[\proj[\feature_k^K]]$ and $\anisotropy^K\coloneqq\anisotropy[\proj[\feature^K]]$.
We have to show in addition to~\cite{EfNe19}   
that $\overline{\anisotropy}^K$ converges strongly to 
$a\coloneqq\anisotropy[\proj[\feature]]$ in $L^\infty((0,1)\times\domain)$.
To this end, we first use the uniform boundedness of $\feature^K$ in $L^\infty([0,1],\featureSpace)$, 
an approximation argument and~\ref{a3} to show the convergence $\overline{\anisotropy}^K-\anisotropy^K\to 0$ in $L^\infty((0,1)\times\domain)$ for $K\to\infty$.
It remains to verify $\anisotropy^K\to\anisotropy$ in $L^\infty((0,1)\times\domain)$ for $K\to\infty$.
The variational inequality
\[
\vert\feature^K(t,\psi_t^K(x))-\feature^K(s,\psi_s^K(x))\vert\leq\int_s^t z^K(r,\psi_r^K(x))\dx r
\]
implies $\feature^K(t)\rightharpoonup\feature(t)$ in $\featureSpace$ for every $t\in[0,1]$ using similar arguments as in step (iv) of~\cite[Theorem~4.1]{BeEf14},
which leads to $\anisotropy^K(t,\cdot)\to\anisotropy(t,\cdot)$ in~$L^\infty(\domain)$ using~\ref{a2}.
We are left to show
\[
\Vert \anisotropy^K(t+\tau,\cdot)-\anisotropy^K(t,\cdot)\Vert_{\featureSpace}\to 0
\]
uniformly in~$K$ and $t$ as $\tau\to 0$, which follows by~\ref{a3} from 
$\Vert\feature^K(t+\tau,\cdot)-\feature^K(t,\cdot)\Vert_{\featureSpace}\to 0$.
This equicontinuity in time is a consequence of the variational inequality,
the uniform boundedness of $z^K$ in $L^2((0,1)\times \domain)$, the uniform boundedness of $\psi^K$ and $(\psi^K)^{-1}$ in $C^{0,\frac12}([0,1],C^{1,\alpha}(\overline\domain))$ for suitable $\alpha>0$,
and an approximation argument for~$\feature^K$.
Then, the actual lower semicontinuity is verified using 
a Taylor expansion of~$\energyDensity$ based on~\ref{W3} to relate the energy density~$W$ with~$L$, where the accumulated remainder is of order~$K^{-\frac{1}{2}}$.

\medskip

\noindent
(ii) \textbf{Recovery sequence.}
Before constructing the recovery sequence, we note that the infimum in~\eqref{eq:pathEnergyDeep} 
is actually attained with an associated pair $(v,z)\in\mathcal{C}(\feature)$, 
which follows from~\cite[Proposition~5.3]{EfNe19} together with Remark~\ref{rem:equi}.

\medskip

\noindent
-- To \emph{construct the recovery sequence} one considers the above pair
$(v,z)\in\mathcal{C}(\feature)$ with an associated flow~$\psi$ and 
defines $\deformation_k^K(x)\coloneqq\psi_{t_{k-1}^K,t_k^K}(x)$, $\feature_k^K(x)\coloneqq\feature(t_k^K,x)$, and $\anisotropy_k^K=\anisotropy[\proj[f_k^K]]$ for $k=1,\ldots,K$, where $\psi_{a,b}(\cdot)=\psi(b,\psi^{-1}(a,\cdot))$
for $a,b\in[0,1]$.

\medskip

\noindent
-- Next, the \emph{identification of the recovery sequence limit} is done, i.e.~one  can show that the extension $\feature^K\coloneqq\interpolfeature^K[\featurevec^K,\defvec^K]$ 
of the time discrete feature vectors $\featurevec^K=(\feature_0,\dots,\feature_K)$ converges to~$\feature$ in $L^2([0,1],\featureSpace)$.
To this end, the discrete flow~$\psi^K$ associated with the time discrete family of deformations~$\defvec^K$ is defined in the same way as in the proof of the $\liminf$--inequality.
Following~\cite{EfNe19}, the convergence $\feature^K\to\feature$ is implied by the variational inequality and the convergence of~$\psi^K$
to the time continuous flow~$\psi$ associated with~$v$ in $C^{0,\alpha}([0,1],C^{1,\alpha}(\overline\domain))$ for a suitable~$\alpha>0$.

\medskip

\noindent
-- Furthermore, we have to \emph{verify the $\limsup$-inequality.} The leading order term of a Taylor expansion of the $k$-th component of the discrete path energy 
\[
\int_\domain\anisotropy_k^K\energyDensity(D\phi_k^K)+\gamma|D^m\phi_k^K|^2\dx x
\]
is given by $K^{-2}\int_\domain L[\anisotropy_k^K,w_k^K,w_k^K]\dx x$, where $w_k^K\coloneqq K(\phi_k^K-\Id)$.
The remainder is of higher order following the argumentation in~\cite{EfNe19}.
Using Jensen's inequality and $w_k^K=\dashint_{t_{k-1}^K}^{t_{k}^K}v(t,\psi_k^K(t,x)) \dx t$ with $\psi_k^K(t,x)\coloneqq\psi_{t_{k-1}^K,t}(x)$ 
we obtain
\begin{equation*}
\int_\domain L[\anisotropy_k^K,w_k^K,w_k^K]\dx x
\leq \int_\domain\dashint_{t_{k-1}^K}^{t_{k}^K}\!\!L[\overline{\anisotropy}^K(t,x),v(t,\psi_k^K(t,x)),v(t,\psi_k^K(t,x))]\dx t\dx x
\end{equation*}
for $\overline{\anisotropy}^K$ defined as before.
Following~\cite{EfNe19}, we can replace $\psi_k^K(t,\cdot)$ by the identity in the limit~$K\to \infty$.
We argue analogously as in the case of the $\liminf$--inequality to show $\overline{\anisotropy}^K\to\anisotropy$ in $L^{\infty}((0,1)\times\domain)$.
Thus, we obtain 
\begin{align*}
&\limsup_{k\to\infty}K\sum_{k=1}^K\int_\domain\anisotropy_k^K\energyDensity(D\deformation_k^K)+\gamma|D^m\deformation_k^K|^2\dx x\\
\leq&\int_0^1\int_\domain L[\anisotropy(t,x),v(t,x),v(t,x)]\dx x\dx t\,.
\end{align*}
Finally, the estimate
\begin{align*}
\limsup_{k\to\infty} K \int_\domain\!\! \vert\feature_{k-1}^K\!-\feature_k^K\circ\deformation_k^K\vert^2\dx x 
\leq \int_0^1\!\!\int_\domain \!\! z^2(t,x)\dx x\dx t
\end{align*}
follows via another application of Jensen's inequality as in~\cite{EfNe19}. 
\end{proof}
This theorem implies the convergence and existence of geodesic paths for the (time continuous) deep feature space metamorphosis model in the following sense:
\begin{theorem}[Convergence of discrete geodesics]
\label{thm:convergence}
Suppose that the assumptions \ref{W1}--\ref{W3} and \ref{a1}--\ref{a3} hold true.
Let $\feature_A,\feature_B\in\featureSpace$  be fixed.
For $K\in\N$ sufficiently large let $\feature^K$ be a minimizer of $\pathenergy^K$ subject to $\feature^K(0)=\feature_A$ and $\feature^K(1)=\feature_B$.
Then, a subsequence of $\{\feature^K\}_{K\in \N}$ weakly converges in $L^2([0,1]\times\featureSpace)$ to a minimizer of the continuous path energy~$\pathenergy$ as $K\to\infty$.
Finally, the associated sequence of discrete path energies converges to the minimal continuous path energy.
\end{theorem}
\begin{proof}
The proof is analogous  to the proof of~\cite[Theorem~5.5]{EfNe19}.
\end{proof}

\section{Fully discrete model in feature space}\label{sec:fullyDiscrete}
In this section, we present the fully discrete deep feature space metamorphosis model on the image domain~$\domain=[0,1]^2$.
We use bold face letters to differentiate discrete feature maps, images,  and deformations (also considered as vectors) from their continuous counterparts.
For $\text{for }M,N\geq 3$ we define the computational domain and its boundary as follows:
\begin{align*}
\discreteDomain&=\{\tfrac{0}{M-1},\ldots,\tfrac{M-1}{M-1}\}\times\{\tfrac{0}{N-1},\ldots,\tfrac{N-1}{N-1}\}\,,\\
\partial\discreteDomain&=\discreteDomain\backslash\{\tfrac{1}{M-1},\ldots,\tfrac{M-2}{M-1}\}\times\{\tfrac{1}{N-1},\ldots,\tfrac{N-2}{N-1}\}\,.
\end{align*}
We define the discrete $L^p$-norm of a discrete feature map $\discreteFeature$ as
\[
\left\Vert\discreteFeature\right\Vert_{L^p(\discreteDomain)}^p=\frac{1}{MN}\sum_{(i,j)\in\discreteDomain}\Vert\discreteFeature(i,j)\Vert_2^p
\]
and the set of admissible deformations is given by
\begin{equation*}
\discreteAdmset=\Big\{ \discreteDeformation:\discreteDomain\to\discreteDomain:
\discreteDeformation=\Id\text{ on }\partial\discreteDomain,\ \det(\nabla_\MN\discreteDeformation)>0\Big\}\,.
\end{equation*}
Furthermore, the discrete Jacobian operator~$\nabla_\MN$ of $\discreteDeformation$ at $(i,j)\in\discreteDomain$ is defined  
as the forward finite difference operator with Neumann boundary conditions. To
further stabilize the computation, the Jacobian operator
applied to the features is approximated using a Sobel filter.
The discrete image space and the discrete feature space are given by~$\imageSpace_{\MN}=\{\Image:\discreteDomain\to\R^3\}$
and $\featureSpace_{\MN}=\{\discreteFeature:\discreteDomain\to\R^{3+C}\}$, respectively.

A numerically reasonable approximation of the spatial warping operator~$\warp$, which approximates
the pullback of a feature channel~$\discreteFeature\circ\discreteDeformation$ at a point $(k,l)\in\discreteDomain$, is given by
\begin{equation*}
\warp[\discreteFeature,\discreteDeformation](k,l)
=\sum_{(i,j)\in\discreteDomain}\mathbf{s}(\discreteDeformation_{1}(k,l)-i)\mathbf{s}(\discreteDeformation_{2}(k,l)-j)\discreteFeature(i,j)\,,
\end{equation*}
where $\mathbf{s}$ is the third order B-spline interpolation kernel.
Then, the fully discrete mismatch functional~$\DataEnergy_{\MN}$ that approximates
$\int_\domain|\tilde{\feature}\circ\deformation-\feature|^2\dx x$ reads as
\begin{equation*}
\DataEnergy_{\MN}[\discreteFeature,\tilde\discreteFeature,\discreteDeformation]
=\frac{1}{2(3+C)}\sum_{c=1}^{3+C}\left\Vert\warp[\tilde\discreteFeature^c,\discreteDeformation]-\discreteFeature^c\right\Vert_{L^2(\discreteDomain)}^2\,.
\end{equation*}
Likewise, the lower order anisotropic regularization functional $\int_\domain\anisotropy\energyDensity(D\deformation)\dx x$ is discretized as follows:
\[
\RegEnergy_{\MN}[\discreteDeformation,\discreteAnisotropy]=\Vert\discreteAnisotropy\energyDensity(\nabla_\MN\discreteDeformation)\Vert_{L^1(\discreteDomain)}\,.
\]
For simplicity, we neglect the $H^m$-seminorm of the deformations.
In the spatially continuous context of the above convergence proof the compactness induced by the 
$H^m$-seminorm turned out to be indispensable.
In the case of the spatial discretization the grid dependent regularity is ensured by the use of cubic B-splines.

In summary, the fully discrete path energy in the deep metamorphosis model for a 
$(K+1)$-tuple $(\discreteFeature_k)_{k=0}^K$ of discrete feature maps, 
a $K$-tuple $(\discreteDeformation_k)_{k=1}^K$ of discrete deformations,
and a $K$-tuple $(\discreteAnisotropy_k)_{k=1}^K$ of discrete anisotropies
reads as
\begin{equation*}
\Pathenergy_{\MN}^{K}[(\discreteFeature_k)_{k=0}^K,(\discreteDeformation_k)_{k=1}^K,(\discreteAnisotropy_k)_{k=1}^K]
=K\sum_{k=1}^K\RegEnergy_{\MN}[\discreteDeformation_k,\discreteAnisotropy_k]+\frac{1}{\delta}\DataEnergy_{\MN}[\discreteFeature_{k-1},\discreteFeature_k,\discreteDeformation_k]\,.
\end{equation*}
Finally, a discrete geodesic path~$(\discreteFeature_k)_{k=0}^{K}$ in feature space on a specific multiscale level of a feature hierarchy 
is a minimizer of $\Pathenergy_{\MN}^{K}$ subject to given discrete boundary data $\discreteFeature_0=\discreteFeature_A$ and $\discreteFeature_K=\discreteFeature_B$.
Here, $\discreteFeature_A=(\eta\Image_A,\FeatureOperator_\MN(\Image_A))$ and $\discreteFeature_B=(\eta\Image_B,\FeatureOperator_\MN(\Image_B))$, where 
$\FeatureOperator_{\MN}:\imageSpace_{\MN}\to\{\discreteFeature:\discreteDomain\to\R^C\}$ denotes the fully discrete feature extraction operator.

\paragraph{Simple RGB model.}
As a first model, we consider the simple image intensity-based feature space with $C=0$, in which the feature space~$\featureSpace_{\MN}$ coincides with the space of RGB images~$\imageSpace_{\MN}$.
Since a direct computation of the deformations on the full grid is numerically instable, we incorporate a multilevel scheme.
Initially, we start on the coarsest computational domain of size $M_{\mathrm{init}}\times N_{\mathrm{init}}$ with $M_{\mathrm{init}}=2^{-(L-1)}M$ and $N_{\mathrm{init}}=2^{-(L-1)}N$ for a given $L>0$
and compute a time discrete geodesic sequence for suitably resized input images~$\Image_A,\Image_B$.
Then, in subsequent prolongation steps, the width and the height of the computational domain are successively doubled and the initial deformations and images are obtained
via a bilinear interpolation of the preceding coarse scale solutions.

\paragraph{Deep feature space.}
In the second model, image features are extracted using the prominent VGG network with 19~layers as presented in~\cite{SiZi15} to incorporate semantic information in image morphing.
The VGG network is particularly designed for localization and classification of objects in natural images and thus the feature decomposition of images is well-suited for semantic matching.
The building blocks of this network are convolutional layers with subsequent ReLU nonlinear activation functions and max pooling layers.
Here, the max pooling layers canonically yield a multiscale semantic decomposition of images.

For a given grid~$\discreteDomain$, the discrete feature maps of the fixed discrete input images~$\Image_A$ and $\Image_B$ are 
$\FeatureOperator_{\MN}[\Image_A]$ and $\FeatureOperator_{\MN}[\Image_B]$, where
the operator $\FeatureOperator_{\MN}$ is the response of the VGG network up to the layer as shown in Table~\ref{tab:VGG}.
The discrete images~$\Image_A$ and $\Image_B$ are downsampled to match the corresponding grid size~$MN$ via a bilinear interpolation.
In contrast to the simple RGB model, only the deformations are prolongated via a bilinear interpolation in the multilevel approach since successive features on different multilevels are not necessarily related.
To stabilize the optimization, the features on each multilevel are first optimized using the prolongated deformations. 
\begingroup
\renewcommand{\arraystretch}{1.2}
\setlength{\tabcolsep}{8pt}
\begin{table}[htb]
\caption{Multiscale decomposition of the VGG network used for the discrete feature extraction operator~$\FeatureOperator_{\MN}$.}
\begin{center}
\begin{tabular}{c | c | c}
$M\times N$ &layer & $C$\\\hline
$512\times 512$ & $\conv_{1,2}$ & 64  \\
$256\times 256$ & $\conv_{2,2}$ & 128 \\
$128\times 128$ & $\conv_{3,4}$ & 256 \\
$64\times 64$ & $\conv_{4,4}$ & 512 \\
$32\times 32$ & $\conv_{5,4}$ & 512 
\end{tabular}
\label{tab:VGG}
\end{center}
\end{table}
\endgroup

\subsection{Numerical optimization}
In what follows, we present the numerical optimization scheme to compute geodesics for the fully discrete deep feature metamorphosis model.
Here, we use a variant of the inertial proximal alternating linearized minimization algorithm (iPALM, \cite{PoSa16}).
Several numerical experiments indicate that a direct gradient based minimization of the data mismatch term~$\DataEnergy_{\MN}$ with respect to the deformations is challenging due to
the sensitivity of the warping operator to small perturbations of the deformations.
Thus, to enhance the stability of the algorithm the warping operator is linearized w.r.t.~the deformation at $\tilde\discreteDeformation\in\discreteAdmset$,
which is chosen as the deformation of the previous iteration step in the algorithm.
To further improve the stability of the algorithm, the linearization is based on the gradient 
$\Lambda_c(\discreteFeature,\tilde\discreteFeature,\tilde\discreteDeformation)=\tfrac{1}{2}(\nabla_\MN\warp[\tilde\discreteFeature^c,\tilde\discreteDeformation]+\nabla_\MN\discreteFeature^c)$, which yields the modified mismatch energy
\begin{equation*}
\widetilde\DataEnergy_{\MN}[\discreteFeature,\tilde\discreteFeature,\discreteDeformation,\tilde\discreteDeformation]
=\frac{1}{2(3+C)}\sum_{c=1}^{3+C}\bigg\Vert\warp[\tilde\discreteFeature^c,\tilde\discreteDeformation]
+\left\langle\Lambda_c(\discreteFeature,\tilde\discreteFeature,\tilde\discreteDeformation),\discreteDeformation-\tilde\discreteDeformation\right\rangle-\discreteFeature^c\bigg\Vert_{L^2(\discreteDomain)}^2\,.
\end{equation*}
The mismatch energy can be efficiently minimized incorporating a proximal mapping, which is defined for a function~$\discreteFeature:\discreteDomain\to(-\infty,\infty]$ for $\tau>0$ as follows:
\[
\operatorname{prox}_\tau^\discreteFeature(i):=\argmin_{j:\discreteDomain\to(-\infty,\infty]}\frac{\tau}{2}\left\Vert i-j\right\Vert_{L^2(\discreteDomain)}^2+\discreteFeature(j)\,.
\]
The proximal operator with respect to the deformation~$\discreteDeformation$ for a fixed~$\tau>0$ is given by 
\begin{align*}
\operatorname{prox}_{\tau}^{\frac{K}{\delta}\widetilde\DataEnergy_{\MN}}[\discreteDeformation]
=&\left(\Id+\tfrac{K}{\tau\delta(3+C)}\sum_{c=1}^{3+C}\Lambda_c(\discreteFeature,\tilde\discreteFeature,\tilde\discreteDeformation)\Lambda_c(\discreteFeature,\tilde\discreteFeature,\tilde\discreteDeformation)^\top\right)^{-1} \\
&\quad\Big(\discreteDeformation-\tfrac{K}{\tau\delta(3+C)}\sum_{c=1}^{3+C}\Big(\Lambda_c(\discreteFeature,\tilde\discreteFeature,\tilde\discreteDeformation)\warp[\tilde\discreteFeature^c,\tilde\discreteDeformation]
-\Lambda_c(\discreteFeature,\tilde\discreteFeature,\tilde\discreteDeformation)\Lambda_c(\discreteFeature,\tilde\discreteFeature,\tilde\discreteDeformation)^\top\tilde\discreteDeformation-\Lambda_c(\discreteFeature,\tilde\discreteFeature,\tilde\discreteDeformation)\discreteFeature^c\Big)\Big)\,,
\end{align*}
where the function values on $\partial\discreteDomain$ remain unchanged.

\begin{algorithm}
\SetInd{1ex}{1ex}
\For{$j=1$ \KwTo $J$}{
\For{$k=1$ \KwTo $K$}{
\tcc{update anisotropy}
$\discreteAnisotropy_k^{[j+1]}=\discreteAnisotropy[\proj[\discreteFeature_k^{[j]}]]$\;
\tcc{update deformation}
$\discreteDeformation_k^{[j+1]}=\operatorname{prox}_{L[\discreteDeformation_k^{[j]}]}^{\frac{K}{\delta}\widetilde\DataEnergy_{\MN}}\left[\widetilde{\discreteDeformation}_k^{[j]}-\tfrac{K}{L[\discreteDeformation_k^{[j]}]}\nabla_{\discreteDeformation_k}
\RegEnergy_{\MN}[\widetilde{\discreteDeformation}_k^{[j]},\discreteAnisotropy_k^{[j+1]}]\right]$\;
\If{$k<K$}{
\tcc{update features}
$\discreteFeature_k^{[j+1]}=\widetilde{\discreteFeature}_k^{[j]}-\tfrac{1}{L[\discreteFeature_k^{[j]}]}\nabla_{\discreteFeature_k}\Pathenergy_{\MN}^K[\widehat{\discreteFeature}_k^{[j]},(\discreteDeformation_1^{[j+1]},\ldots,\discreteDeformation_k^{[j+1]},\discreteDeformation_{k+1}^{[j]},\ldots,\discreteDeformation_K^{[j]}),$
\hspace*{20em}$(\discreteAnisotropy_1^{[j+1]},\ldots,\discreteAnisotropy_k^{[j+1]},\discreteAnisotropy_{k+1}^{[j]},\ldots,\discreteAnisotropy_K^{[j]})]$\;
}
}
}
\caption{Algorithm for minimizing $\Pathenergy_{\MN}^K$ on one multilevel.}
\label{algo:optimization}
\end{algorithm}

Algorithm~\ref{algo:optimization} summarizes the iteration steps for the minimization of the fully discrete path energy~$\Pathenergy_{\MN}^K$, where
for a specific optimization variable~$\discreteFeature$ the extrapolation with $\beta>0$ of the $k^{th}$~path element in the $j^{th}$ iteration step reads as follows:
\begin{align*}
\widetilde{\discreteFeature}_k^{[j]}=\discreteFeature_k^{[j]}+\beta(\discreteFeature_k^{[j]}-\discreteFeature_k^{[j-1]}),\quad
\widehat{\discreteFeature}_k^{[j]}=(\discreteFeature_0,\discreteFeature_1^{[j+1]},\ldots,\discreteFeature_{k-1}^{[j+1]},
\widetilde{\discreteFeature}_k^{[j]},\discreteFeature_{k+1}^{[j]},\ldots,\discreteFeature_{K-1}^{[j]},\discreteFeature_K)\,.
\end{align*}
Here, we use the notation $L[\discreteFeature]$ for the Lipschitz constant of the function $\discreteFeature$, which is determined by backtracking.

\section{Numerical results}\label{sec:results}
In this section, numerical results for both the RGB and the deep feature model are shown.
All parameters used in the computation are specified in Table~\ref{tab:parameters}.
\begin{table}[htb]
\caption{The parameter values for all examples.}
\begin{center}
\begin{tabular}{ c | c | c}
parameter & RGB & deep  \\\hline\hline
\multirow{11}{*}{}
$K$ & \multicolumn{ 2}{c}{15} \\
$\delta$ & \multicolumn{ 2}{c}{1} \\
$L$ & \multicolumn{ 2}{c}{5} \\
$\beta$ & \multicolumn{ 2}{c}{$\tfrac{1}{\sqrt{2}}$}\\
$J$ & \multicolumn{ 2}{c}{250} \\
$\sigma$ & \multicolumn{ 2}{c}{0.5}\\
$\rho$ & \multicolumn{ 2}{c}{2}\\
$\xi_1$ & \multicolumn{ 2}{c}{1000} \\
$\xi_2$ & \multicolumn{ 2}{c}{$10^{-6}$}\\\hline
$\mu$ & 0.025 & 0.002 \\
$\lambda$ & 0.1 & 0.002 \\
$\eta$ &  & $10^{-6}$ \\
\end{tabular}
\end{center}
\label{tab:parameters}
\end{table}
Figure~\ref{fig:vanGoghSequence} depicts the geodesic sequences for two self-portraits by van Gogh\footnote{
public domain,
\url{https://commons.wikimedia.org/wiki/File:Vincent_Willem_van_Gogh_102.jpg};
\url{https://commons.wikimedia.org/wiki/File:SelbstPortrait_VG2.jpg}
}
($M \times N= 496 \times 496$)
for $k\in\{0,3,6,9,12,15\}$ obtained with the RGB model (first row) and the deep feature model (fifth row).
The superiority of the deep model compared to the simple RGB model is exemplarily visualized by the zoom (magnification factor~$4$) of the ear region 
depicted in the second and sixth row.
The remaining rows contain the corresponding sequences of anisotropy weights (third/seventh row) and color-coded displacement fields (fourth/eighth row),
where the hue refers to the direction of the displacements and the intensity is proportional to its norm as indicated by the leftmost color wheel.
Figure~\ref{fig:CatsSequence} presents analogous results for two photos of animals\footnote{
first photo detail by Domenico Salvagnin (CC BY 2.0), \url{https://commons.wikimedia.org/wiki/File:Yawn!!!_(331702223).jpg};
second photo detail by Eric Kilby (CC BY-SA 2.0), \url{https://commons.wikimedia.org/wiki/File:Panthera_tigris_-Franklin_Park_Zoo,_Massachusetts,_USA-8a_(2).jpg}}
for $M\times N=512\times 512$ with a zoom on the mouth region.
Note that the deep model is capable of accurately deforming the carnassial teeth.

\begin{figure}[htb]
\centering
 \includegraphics[width=0.84\linewidth]{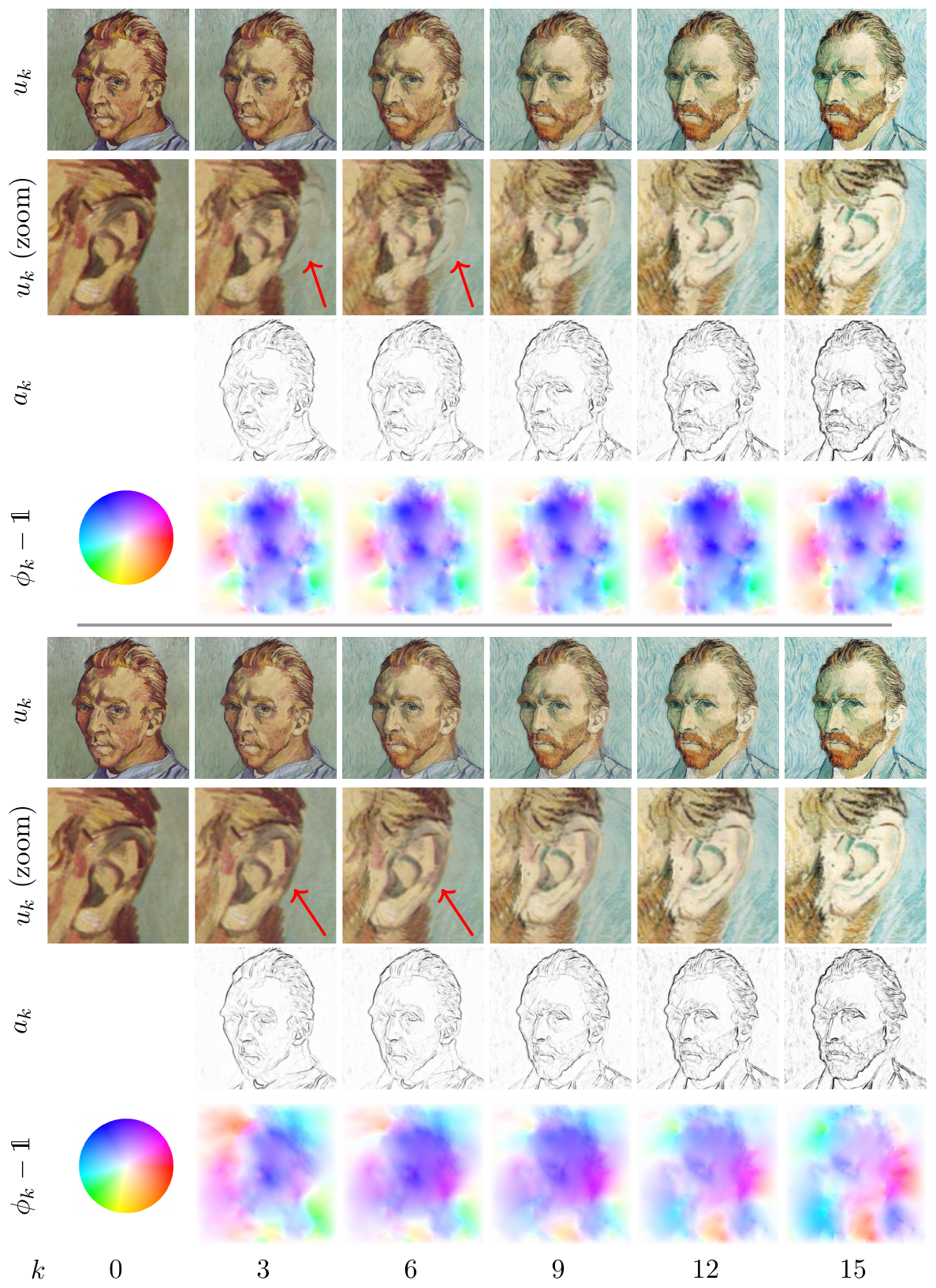}
 \caption{Time discrete geodesic sequences of self-portraits by van Gogh for the RGB feature (first row) and deep feature model (fifth row)
along with a zoom of the ear region with magnification factor~$4$ (second/sixth row)
and the associated sequences of anisotropy weights (third/seventh row) and color-coded displacement fields~$\deformation_k-\Id$ (fourth/eighth row).
Note that the intensity-based approach leads to blending artifacts indicated by the arrows, which are resolved in the deep metamorphosis model.
}
\label{fig:vanGoghSequence}
\end{figure}

\begin{figure}[htb]
\centering
 \includegraphics[width=0.84\linewidth]{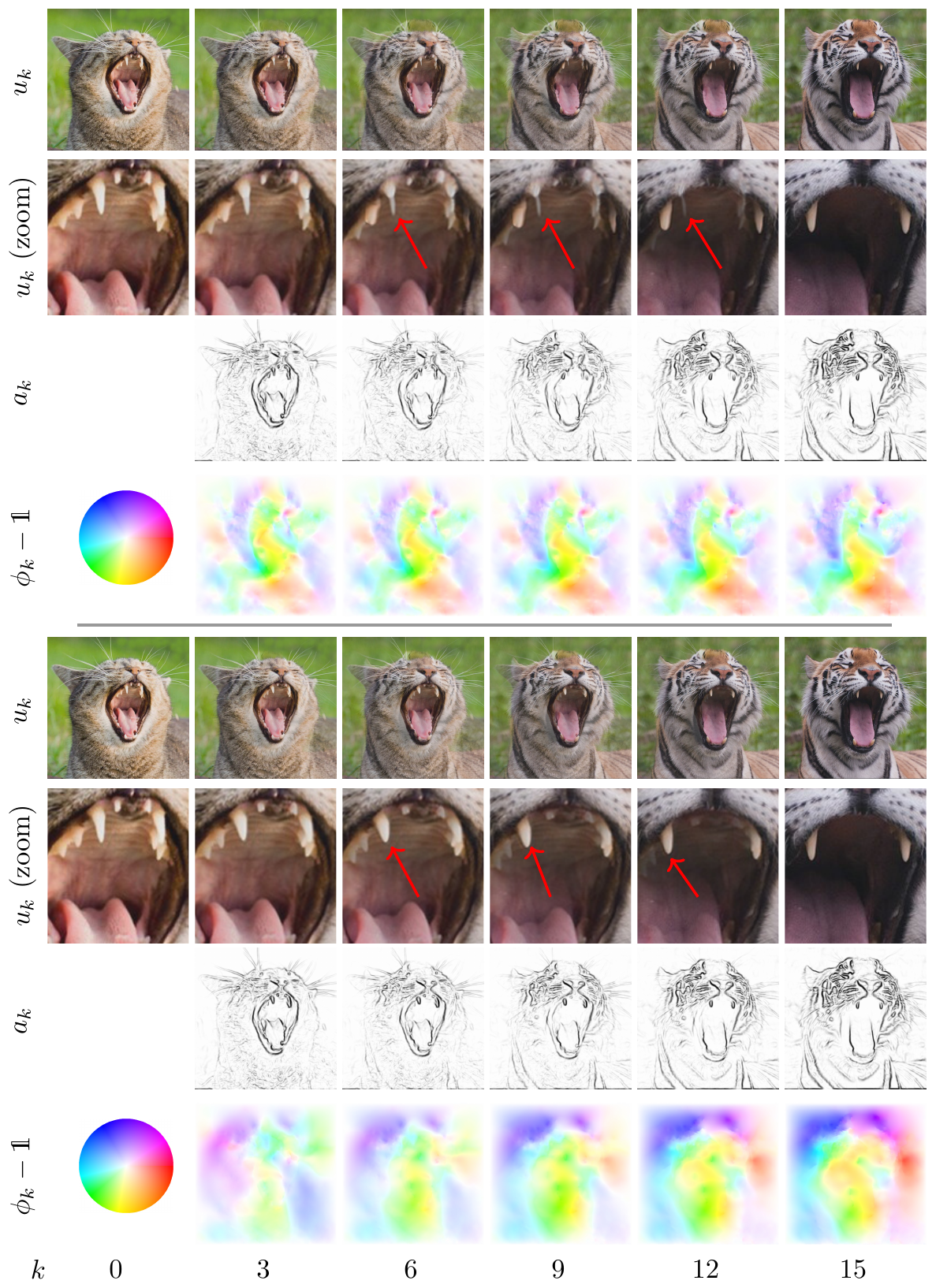}
 \caption{Time discrete geodesic sequences of animal photos for the RGB feature (first row) and deep feature model (fifth row)
along with a zoom of the mouth region with magnification factor~$4$ (second/sixth row)
and the associated sequences of anisotropy weights (third/seventh row) and color-coded displacement fields~$\deformation_k-\Id$ (fourth/eighth row).
Note that the novel deep feature-based model has significantly less blending artifacts as indicated by the arrows.
}
\label{fig:CatsSequence}
\end{figure}

Figure~\ref{fig:presidentsCatherineSequence} shows results of the deep feature model for two paintings of US presidents\footnote{
first painting by Gilbert Stuart (public domain),
\url{https://commons.wikimedia.org/wiki/File:Gilbert_Stuart_Williamstown_Portrait_of_George_Washington.jpg};
second painting by Rembrandt Peale (public domain),
\url{https://commons.wikimedia.org/wiki/File:Thomas_Jefferson_by_Rembrandt_Peale,_1800.jpg}
}
and two portraits of Catherine the Great\footnote{public domain, both portraits by J.~B.~Lampi \url{https://commons.wikimedia.org/wiki/File:Catherine_II_by_J.B.Lampi_(Deutsches_Historisches_Museum).jpg};
\url{https://commons.wikimedia.org/wiki/File:Catherine_II_by_J.B.Lampi_(1780s,_Kunsthistorisches_Museum).jpg}}.
In both cases, the input images have a resolution of $M\times N=512\times 512$.

Finally, we examine the effects of parameter changes of $\xi_1$ and $\delta$.
Figure~\ref{fig:anisotropyVariation} visualizes the anisotropy weight and the deformation field in the RGB model
for a $\xi_1$ value fostering a significantly stronger anisotropy implying much more pronounced jumps in the deformation field (compare with Figure~\ref{fig:vanGoghSequence}).
In addition, Figure~\ref{fig:variationDelta} illustrates the dependency of the resulting morphing sequences on~$\delta$ for the RGB model (first to third row) and the deep model (fourth to sixth row).  As a result, smaller values of $\delta$ lead to less blending. 
Furthermore, the generated geodesic paths using deep features are more robust to changes of~$\delta$ than the RGB model,
which can, for instance, be seen in the cheek or in the eye regions.
\begin{figure}[htb]
\centering
 \includegraphics[width=0.93\linewidth]{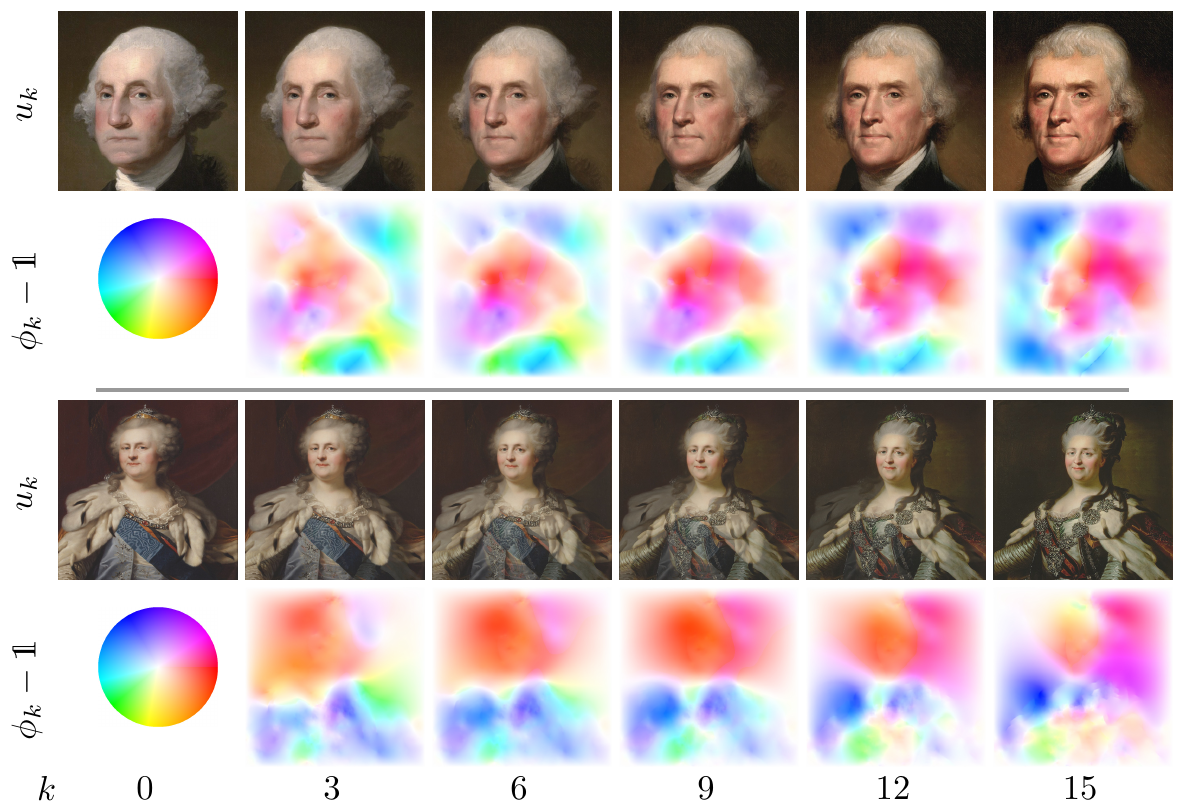}
 \caption{Pairs of time discrete geodesic paths using the deep feature model and corresponding color-coded displacement fields for paintings of US presidents
 (first/second row) as well as for paintings of Catherine the Great (third/forth row).}
\label{fig:presidentsCatherineSequence}
\end{figure}

In all numerical experiments, the displacement fields apparently evolve over time and the involved anisotropy promotes large deformation gradients in the proximity of image interfaces. These are indicated by the sharp interfaces in the color coding of the deformations.
Both models fail to match image regions with no obvious correspondence of the input images, which
can be seen on the cloth regions of the self-portraits, the presidents, and the empress examples,
as well as on parts of the body region and the background in the animal example, where blending artifacts occur.
The deep feature model clearly outperforms the simple RGB model
in regions where the semantic similarity is not reflected by the RGB color features such as the cheek and the ear in the van Gogh example as well as the teeth of the animals.
Moreover, to compute a visually appealing time discrete geodesic sequence, 
a fourth color channel representing a manual segmentation of image regions and a color adaptation of the van Gogh self-portraits was required in \cite{BeEf14}.
This is obsolete in the proposed deep feature based model due to the incorporation of semantic information.

\begin{figure}
\centering
\includegraphics[width=0.6\linewidth]{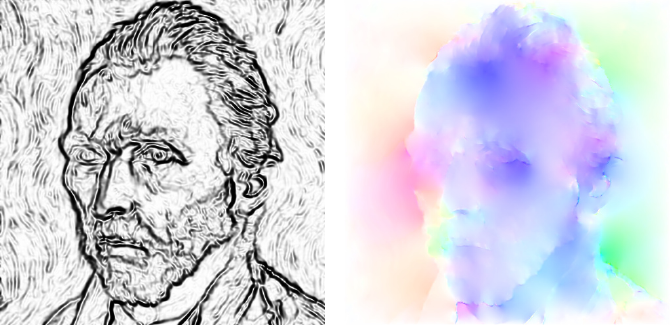}
\caption{Visualization of the anisotropy in RGB model for a significant smaller value $\xi_1=200$ compared to Figure~\ref{fig:vanGoghSequence}: anisotropy operators (left) and color-coded displacement fields (right) for $k=12$.
}
\label{fig:anisotropyVariation}
\end{figure}



\begin{figure}
\centering
\includegraphics[width=\linewidth]{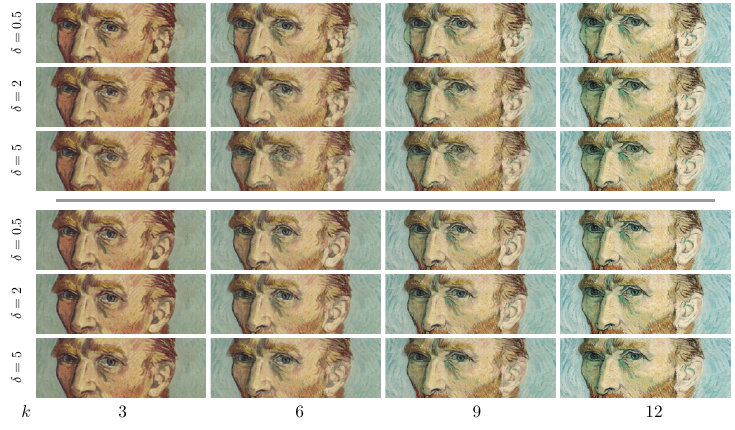}
\caption{Variation of the parameter $\delta$ for RGB model (first to third row) and deep feature model (fourth to sixth row).}
\label{fig:variationDelta}
\end{figure}

\subsection*{Acknowledgements} 
Alexander Effland, Erich Kobler and Thomas Pock acknowledge support from the European Research Council under the Horizon 2020 program, ERC starting grant HOMOVIS (No. 640156).
Marko Rajkovi\'c and Martin Rumpf acknowledge support of the Collaborative Research Center 1060 fun\-ded by 
the Deutsche Forschungsgemeinschaft (DFG, German Research Foundation)  
and the Hausdorff Center for Mathematics, funded by the DFG
under Germany's Excellence Strategy - GZ 2047/1, Projekt-ID 390685813.

\bibliographystyle{alpha}
\bibliography{JMIV}

\begin{thebibliography}{BMTY05}

\bibitem[Arn66]{Ar66a}
Vladimir Arnold.
\newblock Sur la g\'{e}om\'{e}trie diff\'{e}rentielle des groupes de {L}ie de
  dimension infinie et ses applications \`a l'hydrodynamique des fluides
  parfaits.
\newblock {\em Ann. Inst. Fourier (Grenoble)}, 16(fasc., fasc. 1):319--361,
  1966.

\bibitem[Bal81]{Ba81}
John~M. Ball.
\newblock Global invertibility of {S}obolev functions and the interpenetration
  of matter.
\newblock {\em Proc. Roy. Soc. Edinburgh Sect. A}, 88(3-4):315--328, 1981.

\bibitem[BER15]{BeEf14}
Benjamin Berkels, Alexander Effland, and Martin Rumpf.
\newblock Time discrete geodesic paths in the space of images.
\newblock {\em SIAM J. Imaging Sci.}, 8(3):1457--1488, 2015.

\bibitem[BMR13]{BuMo13}
Martin Burger, Jan Modersitzki, and Lars Ruthotto.
\newblock A hyperelastic regularization energy for image registration.
\newblock {\em SIAM J. Sci. Comput.}, 35(1):B132--B148, 2013.

\bibitem[BMTY05]{BeMiTr05}
M.~Faisal Beg, Michael~I. Miller, Alain Trouv{\'e}, and Laurent Younes.
\newblock Computing large deformation metric mappings via geodesic flows of
  diffeomorphisms.
\newblock {\em International Journal of Computer Vision}, 61(2):139--157, Feb
  2005.

\bibitem[CCT18]{CCT16}
Nicolas Charon, Benjamin Charlier, and Alain Trouv\'{e}.
\newblock Metamorphoses of functional shapes in {S}obolev spaces.
\newblock {\em Found. Comput. Math.}, 18(6):1535--1596, 2018.

\bibitem[Cia88]{Ci88}
Philippe~G. Ciarlet.
\newblock {\em Mathematical elasticity. {V}ol. {I}}, volume~20 of {\em Studies
  in Mathematics and its Applications}.
\newblock North-Holland Publishing Co., Amsterdam, 1988.
\newblock Three-dimensional elasticity.

\bibitem[DGM98]{DuGrMi98}
Paul Dupuis, Ulf Grenander, and Michael~I. Miller.
\newblock Variational problems on flows of diffeomorphisms for image matching.
\newblock {\em Quart. Appl. Math.}, 56(3):587--600, 1998.

\bibitem[DM93]{dalMaso93}
Gianni Dal~Maso.
\newblock {\em An introduction to {$\Gamma$}-convergence}, volume~8 of {\em
  Progress in Nonlinear Differential Equations and their Applications}.
\newblock Birkh\"{a}user Boston, Inc., Boston, MA, 1993.

\bibitem[DR04]{DrRu04}
Marc Droske and Martin Rumpf.
\newblock A variational approach to nonrigid morphological image registration.
\newblock {\em SIAM J. Appl. Math.}, 64(2):668--687, 2003/04.

\bibitem[Eff18]{Ef18}
Alexander Effland.
\newblock {\em Discrete {Riemannian} Calculus and A Posteriori Error Control on
  Shape Spaces}.
\newblock PhD thesis, University of Bonn, 2018.

\bibitem[EKPR19]{EfKo19}
Alexander Effland, Erich Kobler, Thomas Pock, and Martin Rumpf.
\newblock Time discrete geodesics in deep feature spaces for image morphing.
\newblock In {\em Scale Space and Variational Methods in Computer Vision},
  pages 171--182, Cham, 2019. Springer International Publishing.

\bibitem[ENR20]{EfNe19}
Alexander Effland, Sebastian Neumayer, and Martin Rumpf.
\newblock Convergence of the time discrete metamorphosis model on {Hadamard}
  manifolds.
\newblock {\em SIAM J. Imaging Sci.}, 13(2):557--588, 2020.

\bibitem[JM00]{JoMi00}
Sarang~C. Joshi and Michael~I. Miller.
\newblock Landmark matching via large deformation diffeomorphisms.
\newblock {\em IEEE Trans. Image Process.}, 9(8):1357--1370, 2000.

\bibitem[KSH12]{KrSu12}
Alex Krizhevsky, Ilya Sutskever, and Geoffrey~E Hinton.
\newblock {ImageNet} classification with deep convolutional neural networks.
\newblock In F.~Pereira, C.~J.~C. Burges, L.~Bottou, and K.~Q. Weinberger,
  editors, {\em Advances in Neural Information Processing Systems 25}, pages
  1097--1105. Curran Associates, Inc., 2012.

\bibitem[Mos69]{Mosco69}
Umberto Mosco.
\newblock Convergence of convex sets and of solutions of variational
  inequalities.
\newblock {\em Advances in Math.}, 3:510--585, 1969.

\bibitem[MTY02]{MiTrYo02}
Michael~I. Miller, Alain Trouvé, and Laurent Younes.
\newblock On the metrics and euler-lagrange equations of computational anatomy.
\newblock {\em Annual Review of Biomedical Engineering}, 4(1):375--405, 2002.
\newblock PMID: 12117763.

\bibitem[MTY15]{MTY15}
Michael~I. Miller, Alain Trouv\'e, and Laurent Younes.
\newblock Hamiltonian systems and optimal control in computational anatomy: 100
  years since d'{A}rcy {T}hompson.
\newblock {\em Annual Review of Biomed. Eng.}, 17(1):447--509, 2015.

\bibitem[MY01]{MiYo01}
Michalel~I. Miller and Laurent Younes.
\newblock Group actions, homeomorphisms, and matching: A general framework.
\newblock {\em International Journal of Computer Vision}, 41(1):61--84, Jan
  2001.

\bibitem[Nir66]{Ni66}
Louis Nirenberg.
\newblock An extended interpolation inequality.
\newblock {\em Ann. Scuola Norm. Sup. Pisa Cl. Sci. (3)}, 20:733--737, 1966.

\bibitem[NPS18]{NePe18}
Sebastian Neumayer, Johannes Persch, and Gabriele Steidl.
\newblock Morphing of manifold-valued images inspired by discrete geodesics in
  image spaces.
\newblock {\em SIAM J. Imaging Sci.}, 11(3):1898--1930, 2018.

\bibitem[Nv91]{NeSi91}
Jind\v{r}ich Ne\v{c}as and Miroslav \v{S}ilhav\'{y}.
\newblock Multipolar viscous fluids.
\newblock {\em Quart. Appl. Math.}, 49(2):247--265, 1991.

\bibitem[PM90]{PeMa90}
Pietro Perona and Jitendra Malik.
\newblock Scale-space and edge detection using anisotropic diffusion.
\newblock {\em IEEE Transactions on Pattern Analysis and Machine Intelligence},
  12(7):629--639, 1990.

\bibitem[PS16]{PoSa16}
Thomas Pock and Shoham Sabach.
\newblock Inertial proximal alternating linearized minimization (i{PALM}) for
  nonconvex and nonsmooth problems.
\newblock {\em SIAM J. Imaging Sci.}, 9(4):1756--1787, 2016.

\bibitem[RW15]{RuWi12b}
Martin Rumpf and Benedikt Wirth.
\newblock Variational time discretization of geodesic calculus.
\newblock {\em IMA J. Numer. Anal.}, 35(3):1011--1046, 2015.

\bibitem[RY13]{RY13}
Casey~L. Richardson and Laurent Younes.
\newblock Computing metamorphoses between discrete measures.
\newblock {\em J. Geom. Mech.}, 5(1):131--150, 2013.

\bibitem[RY16]{RY16}
Casey~L. Richardson and Laurent Younes.
\newblock Metamorphosis of images in reproducing kernel {H}ilbert spaces.
\newblock {\em Adv. Comput. Math.}, 42(3):573--603, 2016.

\bibitem[SZ14]{SiZi15}
Karen Simonyan and Andrew Zisserman.
\newblock Very deep convolutional networks for large-scale image recognition.
\newblock {\em CoRR}, abs/1409.1556, 2014.

\bibitem[TY05a]{TrYo05a}
Alain Trouv\'{e} and Laurent Younes.
\newblock Local geometry of deformable templates.
\newblock {\em SIAM J. Math. Anal.}, 37(1):17--59, 2005.

\bibitem[TY05b]{TrYo05}
Alain Trouv\'{e} and Laurent Younes.
\newblock Metamorphoses through {L}ie group action.
\newblock {\em Found. Comput. Math.}, 5(2):173--198, 2005.

\bibitem[VRRC12]{VRRC12}
Fran\c{c}ois-Xavier Vialard, Laurent Risser, Daniel Rueckert, and Colin~J.
  Cotter.
\newblock Diffeomorphic 3{D} image registration via geodesic shooting using an
  efficient adjoint calculation.
\newblock {\em Int. J. Comput. Vis.}, 97(2):229--241, 2012.

\bibitem[VS09]{VS09}
Fran\c{c}ois-Xavier Vialard and Filippo Santambrogio.
\newblock Extension to {BV} functions of the large deformation diffeomorphisms
  matching approach.
\newblock {\em C. R. Math. Acad. Sci. Paris}, 347(1-2):27--32, 2009.

\bibitem[You10]{Younes2010}
Laurent Younes.
\newblock {\em Shapes and diffeomorphisms}, volume 171 of {\em Applied
  Mathematical Sciences}.
\newblock Springer-Verlag, Berlin, 2010.

\end{thebibliography}

\end{document}